\documentclass[10pt,reqno]{amsart}
\usepackage{latexsym,amssymb,amsmath}
\usepackage{amsmath,amsfonts}
\usepackage[hidelinks]{hyperref}
\usepackage{epsfig}
\usepackage{graphicx}
\usepackage{enumerate}

\usepackage{xcolor}
\usepackage{cancel}
\usepackage{scalerel,stackengine}
\stackMath
\newcommand\reallywidehat[1]{%
\savestack{\tmpbox}{\stretchto{%
  \scaleto{%
    \scalerel*[\widthof{\ensuremath{#1}}]{\kern-.6pt\bigwedge\kern-.6pt}%
    {\rule[-\textheight/2]{1ex}{\textheight}}
  }{\textheight}%
}{0.5ex}}%
\stackon[1pt]{#1}{\tmpbox}%
}

\DeclareMathOperator{\R}{\mathbb{R}}
\DeclareMathOperator{\N}{\mathbb{N}}

\newcommand{\norm}[1]{\left\lVert#1\right\rVert}
\theoremstyle{plain}
\newtheorem{theorem}{Theorem}[section]

\newtheorem{Corollary}[theorem]{Corollary}
\newtheorem{lemma}[theorem]{Lemma}

\newtheorem{proposition}[theorem]{Proposition}

\theoremstyle{remark}
\newtheorem{Remark}[theorem]{Remark}

\allowdisplaybreaks

\newcommand{\Hmm}[1]{\leavevmode{\marginpar{\tiny%
$\hbox to 0mm{\hspace*{-0.5mm}$\leftarrow$\hss}%
\vcenter{\vrule depth 0.1mm height 0.1mm width \the\marginparwidth}%
\hbox to 0mm{\hss$\rightarrow$\hspace*{-0.5mm}}$\\\relax\raggedright #1}}}

\begin{document}

\title[Long-time behavior of SgKdV with additive noise]{Global Well-posedness and Scattering for Stochastic generalized KdV Equations with additive noise}

\author{Engin Ba\c{s}ako\u{g}lu}

\author{Faruk Temur}

\author{O\u{g}uz Y{\i}lmaz}

\address{Institute of Mathematical Sciences, ShanghaiTech University, Shanghai, 201210, China}

\email{ebasakoglu@shanghaitech.edu.cn}

\address{Department of Mathematics,
Izmir Institute of Technology, 
Urla 35430, Izmir, Turkey}

\email{faruktemur@iyte.edu.tr}

\address{Department of Mathematics,
Bo\u gazi\c ci University, 
Bebek 34342, Istanbul, Turkey}

\email{oguz.yilmaz@bogazici.edu.tr}

\subjclass[2010]{}

\begin{abstract}
    We study the defocusing stochastic generalized Korteweg–de Vries equations (sgKdV) driven by additive noise, with a focus on mass-critical and supercritical nonlinearities. For integers $k\geq 4$, we establish local well-posedness almost surely up to scaling critical regularity. We also prove global well-posedness and scattering in $L^{2}_{x}(\R)$ for the mass-critical equation with small initial data; also in $H^{1}_{x}(\R)$ for the mass supercritical equation. In particular, we prove oscillatory integral estimates associated with more general dispersion relations, which are of independent interest; and we make use of a special case of these estimates as a main ingredient for the necessary bounds on the tail of the stochastic convolution for sgKdV, which is crucial to conclude scattering results.
\end{abstract}
\maketitle

\section{Introduction}
In this paper, we consider the defocusing stochastic gKdV equations with additive stochastic forcing
\begin{equation}\label{eq:stochastic_gKdV}
    \begin{cases}
        &du+u_{xxx}dt=(u^{k+1})_{x}dt+dW,\\
        &u(0,x)=u_{0}(x),
    \end{cases}
    \quad t\geq 0,\, x\in\mathbb{R},
\end{equation}
for even integers $k\geq 4$, and $dW=g(\omega,t)\phi(x)dB(\omega,t)$, where $\phi$ is Schwartz-in-space function, $g$ is a real-valued predictable process with $g\in L^{2}_{t}([0,\infty))$, $\mathbb{P}$-a.s., $g\in L^{\infty}_{\omega,t}(\Omega\times[0,T])$ for $T>0$, and $B(\omega,t)$ is a real-valued Brownian motion on a probability space $(\Omega,\{\mathcal{F}_{t}:\,t\geq0\},\mathbb{P})$ with right continuous filtration.

Such equations as in \eqref{eq:stochastic_gKdV} model the evolution of unidirectional, weakly nonlinear, and weakly dispersive waves \cite{Korteweg_First_article,Segur_Hammack_1982}. In particular, when $k=2$, the equation \eqref{eq:stochastic_gKdV} refers to the modified KdV equation, which arises as a model for large amplitude internal waves in a density stratified medium, and for Fermi–Pasta-Ulam lattices with bistable nonlinearity, see \cite{johnson_KdV_physical}. Unlike the KdV and modified KdV equations, the deterministic analogue of \eqref{eq:stochastic_gKdV} with $k\geq 3$ is not completely integrable, see for instance \cite{Miura_KdV_integrability}. Hovewer, there are still conserved quantities, such as mass 
\begin{equation}\label{mass_functional}
    M(u(t)):=\int_{\mathbb{R}}u^{2}(t,x)dx,
\end{equation}
and the energy
\begin{equation}\label{energy_functional}
    E(u(t)):=\frac{1}{2}\int_{\mathbb{R}}u_{x}^{2}(t,x)dx+\frac{1}{k+2}\int_{\mathbb{R}}u^{k+2}(t,x)dx.
\end{equation}

The well-posedness and scattering results for the deterministic analogue of \eqref{eq:stochastic_gKdV}, especially for $k\geq 3$, were first established by the pioneering work of Kenig-Ponce-Vega \cite{Kenig_gKdV_1993}. They proved local and global well-posedness in $H^{s}_{x}(\mathbb{R})$ for $s\geq s_{k}$ with $k\geq 4$, where the scaling critical exponent $s_k$ associated with \eqref{eq:stochastic_gKdV} is given by \begin{align}\label{scalingcr}
  s_{k}=\frac{k-4}{2k} . 
\end{align}
In particular, when $k=4$, solutions for the deterministic mass-critical gKdV equation \eqref{eq:deterministic_gKdV_article} satisfy the following global existence result: 
\begin{theorem}[Theorem 2.8 in \cite{Kenig_gKdV_1993}]\label{thm:scat_det_mass_crit_gKdV}
    For any initial data $u_{0}\in L^{2}_{x}(\mathbb{R})$ with $\norm{u_{0}}_{L^{2}_{x}}<\delta$ for some $\delta>0$, the solution $y$ to \eqref{eq:deterministic_gKdV_article} with k=4 exists globally and satisfies
\begin{equation*}
\begin{aligned}
    &y\in C^{1}_{x}L^{2}_{t}(\mathbb{R}\times\mathbb{R})\cap C^{0}_{t}L^{2}_{x}(\mathbb{R}\times\mathbb{R}),\\& \norm{y}_{L^{5}_{x}L^{10}_{t}(\mathbb{R}\times\mathbb{R})}+\norm{\partial_{x}y}_{L^{\infty}_{x}L^{2}_{t}(\mathbb{R}\times\mathbb{R})}<\infty.
    \end{aligned}
\end{equation*}
\end{theorem}
Later, in \cite{Dodson_gKdV}, Dodson removed the assumption of small initial data from the above result, and using a concentration compactness argument, he proved the existence of a unique strong global-in-time solution to \eqref{eq:deterministic_gKdV_article} with $k=4$ as well as scattering (see \cite[Theorems 1.2, 1.3]{Dodson_gKdV}). In the context of the stochastic equation \eqref{eq:stochastic_gKdV}; however, it is unclear whether an adaptation of a concentration compactness argument works, as the method relies heavily on the conserved quantities of the equation; see, for example, \cite[Lemma 3.1]{Dodson_gKdV}. When $k>4$ is an even integer, the energy scattering for the defocusing problem was established for any $H^{1}_{x}(\R)$ data in \cite{Farah_H1_scattering}. The global well-posedness below the energy level (in the case $k\geq 4$) was studied in \cite{Farah_critical_gKdV,Farah_supercrit_gKdV}. In particular, when $k=4$, it was shown in \cite{Farah_critical_gKdV} that the focusing problem (\eqref{eq:deterministic_gKdV_article} with $-(y^{k+1})_x dt$ on the right side instead of $(y^{k+1})_x dt$) is globally well-posed in Sobolev spaces $H^s$ with $s>\frac{3}{5}$ under a suitable smallness assumption on the initial data, where they utilized an argument of \cite{colliander2002almost}. Also, when $k>4$, the authors \cite{Farah_supercrit_gKdV} proved the global existence of solutions
\begin{itemize}
    \item[(i)] under smallness assumption on the $H^{1}_{x}$ initial data (for the focusing case), 
    \item[(ii)] in $H^{s}_{x}(\mathbb{R})$ with $s>\frac{4(k-1)}{5k}$ for even $k>4$ (for the defocusing case).
\end{itemize}

As for the stochastic equation \eqref{eq:stochastic_gKdV} with additive noise, there are only a few related works initiated by Bouard and Debussche in a series of papers \cite{deBouard_SKdV_1,deBouard_SKdV_3,deBouard_SKdV_2} concerning the case $k=1$. Moreover, the cases $k=2$ and $k=3$ with additive stochastic forcing term are considered in \cite{Millet_SKdV}, in which the almost sure global well-posedness result in the energy space $H^{1}_{x}$ was obtained. Also in the context of global existence and scattering for stochastic nonlinear Schr\"odinger equations, we refer to the papers \cite{Basakoglu_Scattering_NLS, cheung2020conservation, Fan_Zhao, Herr_2019, Oh_Critical} and references therein.

Our first result for equation \eqref{eq:stochastic_gKdV} is as follows:
\begin{theorem}\label{thm:well-posedness}
Let $0<T<\infty$, and assume that $g\in L^{\infty}_{\omega,t}(\Omega\times[0,T])$. Then, we have the following:
    \begin{enumerate}
  \item[$(i)$] Let $k\geq 4$ be an integer and $s\geq s_{k}=\frac{k-4}{2k}$. Then, the initial value problem \eqref{eq:stochastic_gKdV} with the data $u_0\in H^{s}_{x}(\mathbb{R})$ is $\omega$ almost surely locally well-posed.  
       \item[$(ii)$] If $k=4$, the initial value problem \eqref{eq:stochastic_gKdV} is $\omega$ almost surely globally well-posed in $L^{2}_{x}(\mathbb{R})$ for small initial data; if $k>4$ is an even integer, global well-posedness holds in $H^{1}_{x}(\mathbb{R})$ for large data $\omega$ almost surely.
    \end{enumerate}
\end{theorem}
Here we briefly sketch the idea of proof of Theorem \ref{thm:well-posedness}; for details, see Section \ref{section:well-posedness}. First, we note that the solutions to \eqref{eq:stochastic_gKdV} satisfy the following integral formulation:
\begin{equation}\label{Duhamel_of_stochastic_solution}
    u(t)=V(t)u_{0}+\int_{0}^{t}V(t-t')(u^{k+1})_{x}(t')dt'+z(t),
\end{equation}
where
\begin{equation}\label{airy_propagator}
    V(t)=\int_{\R}e^{i(x\xi+t\xi^{3})}d\xi
\end{equation}
is the Airy propagator, and we denote the stochastic convolution (by suppressing the $\omega$ dependence)
\begin{align}\label{stconv}
 z(t) = \int_{0}^{t}V(t-t')dW(t')=\int_{0}^{t}V(t-t')\phi(x)g(t)dB(t').   
\end{align}
Concerning the proof of $(i)$ in Theorem \ref{thm:well-posedness}, the main issue is to derive a priori bounds on the stochastic convolution (which is discussed in Section \ref{A priori estimates}) via the tools from Section \ref{section: preliminaries} and some interpolation arguments discussed in \cite{Millet_SKdV}. This will later yield a control on the key norms needed for the contraction argument in suitable resolution spaces, see \eqref{Y_space} and \eqref{Z_space} for the definitions of such spaces. As for the proof of $(ii)$ in Theorem \ref{thm:well-posedness}, we exploit the sign definiteness of the deterministic energy \eqref{energy_functional} (for $k$ even) to control the growth of mass and energy, despite the lack of exact conservation in the stochastic setting. In Section \ref{glexst}, applying It\^{o}'s lemma to a suitable power of a combination of mass and energy and using well-known estimates, we bound drift and martingale terms uniformly in time. A stochastic Gronwall-type lemma (Lemma \ref{gronwall}) then yields moment bounds that, combined with the local theory, extend solutions globally.

\subsection{The tail of the stochastic convolution and the scattering results}\label{tss}
In our discussion, the scattering property of solutions to equations \eqref{eq:stochastic_gKdV} depends very much on the tail of the stochastic convolution which is defined by
\begin{equation}\label{tail_of_stochastic_convolution}
		z_*(t):=\int_{t}^{\infty}V(t-t')dW(t').
\end{equation}
For a discussion on the stochastic convolution $z(t)$ (see \eqref{stconv}) and the tail of the stochastic convolution $z_{*}(t)$, we refer the readers to \cite{Basakoglu_Scattering_NLS}. In \cite{Basakoglu_Scattering_NLS} we proved $\omega$ a.s. explicit time decay results for $L^p(\R^n)$ norms of $z_*(t)$. Here, in addition,
we prove estimates on $\|D^{\alpha}z_*\|_{L^p_xL^q_t([T,\infty))}$ where $(\alpha, p, q)$ are Kato triples.

Before recalling the Da Prato-Debussche trick for solutions of \eqref{eq:stochastic_gKdV} in what follows, we state the deterministic gKdV equation (to be used in the sequel)
\begin{equation}\label{eq:deterministic_gKdV_article}
    dy +  y_{xxx} dt = (y^{k+1})_{x} dt,
\end{equation}
where $k\geq 4$ is an even integer. Let $T\gg 1$ be a large time parameter and decompose the solution \eqref{Duhamel_of_stochastic_solution} as $u(t)=u_{*}(t)+z_{*}(t)$ (known as the Da Prato-Debussche trick), for $t\geq T$. Then $u_{*}(t)$ satisfies the following random differential equation:
\begin{equation}\label{random_dif_eqn_u_*}
    du_{*}+(u_{*})_{xxx}dt=((u_{*}+z_{*})^{k+1})_{x}dt,\quad (t,x)\in[T,\infty)\times\R.
\end{equation}
Set $u_{*}(T)=y(T)$. Note that for each $\omega\in\Omega$, we pick a different $y(T)$, and turn \eqref{eq:deterministic_gKdV_article} into a random differential equation. Next, we consider $v:=u_{*}-y$ (so $v(T)=0$), and in view of \eqref{eq:deterministic_gKdV_article} and \eqref{random_dif_eqn_u_*}, observe that $v$ satisfies the following random IVP:
\begin{equation}\label{eq:differ_rand_and_det_soln_IVP}
    \begin{cases}
        &dv + v_{xxx}dt = (( v + y + z_{*})^{k+1})_{x}dt - (y^{k+1})_{x}dt\\
        & v(T)=0,
    \end{cases}
    \quad (t,x)\in[T,\infty)\times\mathbb{R}.
\end{equation}
We remark that $u(t)$, $u_{*}(t)$ and $y(t)$ will have the same large time behavior once we establish that the tail of the stochastic convolution $z_{*}(t)$ belongs to the global-in-time Strichartz and Kato type spaces (see Lemmas \ref{lemma:strichartz} and \ref{kato_estimates} for the definition of these spaces), and this helps in proving that the suitable Strichartz and Kato type norms of $v(t)$ on $[T,\infty)$ tend to $0$ as $T\to\infty$, see Section \ref{section:scat_in_L2} for details. This is guaranteed by the assumption of $\omega$ almost sure time-decay of $g(t)$ dictated in Theorems \ref{thm:scat_in_L2} and \ref{thm:scat_in_H1}, and the estimates concerning $z_{*}(t)$ will be discussed in Section \ref{section:est_on_tail}. 

We next state our scattering results. The first one concerns the almost sure scattering for the mass-critical equation, \eqref{eq:stochastic_gKdV} with $k=4$. It is known that the $L^{5}_{x}L^{10}_{t}$ norm is of great importance (which is called the scattering size of solutions), and the key step in the mass critical case in particular is to establish a global-in-time bound for the $L^{5}_{x}L^{10}_{t}$ norm of the solutions to equation \eqref{eq:stochastic_gKdV}. 
To deduce our scattering result in $L^2_x$, we will take advantage of the global well-posedness of the deterministic mass critical gKdV equation \eqref{eq:deterministic_gKdV_article} in $L_{x}^{2}$ with $\norm{y}_{L^{5}_{x}L^{10}_{t}(\R\times\R)}<\infty$ (Theorem \ref{thm:scat_det_mass_crit_gKdV}) together with almost sure global well-posedness of the stochastic mass critical gKdV equation \eqref{eq:stochastic_gKdV} in $L^{2}_{x}$ with a small initial data assumption (Theorem \ref{thm:well-posedness}), and utilize the estimates on the tail of the stochastic convolution $z_{*}(t)$ (Proposition \ref{mainprop}).
\begin{theorem}\label{thm:scat_in_L2}
    Let $k=4$ in \eqref{eq:stochastic_gKdV}. Assume that $\vert g(t)\vert=o(t^{-\gamma})$ $\omega$-a.s. for $\gamma>\frac{2}{3}$ and $\norm{u_{0}}_{L^{2}_{x}}<\delta$ for some $\delta>0$. Then the global $L^{2}_{x}$ solution $u(t)$ of \eqref{eq:stochastic_gKdV}, with initial data $u_{0}$, scatters $\omega$ almost surely forward in time, i.e., $\omega$ almost surely there is $y_{+}\in L^{2}_{x}(\mathbb{R})$ such that
    \begin{equation*}
        \lim_{t\to\infty}\norm{V(-t)u(t)-y_{+}}_{L^{2}_{x}}=0.
    \end{equation*}
   \end{theorem}
For the mass supercritical case, the next theorem provides the almost sure scattering in the energy space. In this case, the scattering size of the solutions is measured by the norm $L^{\frac{5k}{4}}_{x}L^{\frac{5k}{2}}_{t}$. Hence, the crucial part in establishing the scattering result is to bound the solutions to equation $\eqref{eq:stochastic_gKdV}$ with $k>4$ under the norm $L^{\frac{5k}{4}}_{x}L^{\frac{5k}{2}}_{t}$. This in particular will be performed by showing that the difference $v(t)$ (a solution to \eqref{eq:differ_rand_and_det_soln_IVP}) of random and deterministic solutions becomes arbitrarily small as time gets larger. More precisely, the task will be to deduce that
 \begin{equation*}
     \norm{\langle\partial_{x}\rangle v}_{L^{5}_{x}L^{10}_{t}([T,\infty))}+\norm{v}_{L^{\frac{5k}{4}}_{x}L^{\frac{5k}{2}}_{t}([T,\infty))}+\norm{\partial_{x}^{2}v}_{C^{0}_{x}L^{2}_{t}([T,\infty))}+\norm{v}_{C^{0}_{t}H^{1}_{x}([T,\infty))}\to 0
 \end{equation*}
as $T\to\infty$, which implies the pathwise equivalence of the asymptotic behavior of stochastic solutions and deterministic solutions. As in the mass critical case, Proposition \ref{mainprop} will be of utmost importance in the nonlinear estimates to justify the global in time boundedness of Kato admissible norms of the tail of the stochastic convolution $z_{*}(t)$.
\begin{theorem}\label{thm:scat_in_H1}
    Let $k>4$ be an even integer in \eqref{eq:stochastic_gKdV}, $u_{0}\in H^{1}_{x}(\R)$, and assume that $g(t)$ satisfies the same decay condition as in Theorem \ref{thm:scat_in_L2}. Then the global $H^{1}_{x}$ solution $u(t)$ of \eqref{eq:stochastic_gKdV}, with initial data $u_{0}$, scatters $\omega$ almost surely forward in time, i.e., $\omega$ almost surely there is $y_{+}\in H^{1}_{x}(\mathbb{R})$ such that
    \begin{equation*}
        \lim_{t\to\infty}\norm{V(-t)u(t)-y_{+}}_{H^{1}_{x}}=0.
    \end{equation*}
   \end{theorem}
\begin{Remark}\label{remark:oscillatory_integrals}
In proving these theorems, the Kato estimates  play an important role, and it is known that the oscillatory integrals of the form
	\begin{equation}\label{osc1}
 	I^{b,\alpha}(x):=\begin{aligned}
 		\int_{\R} e^{i(\xi^b+x\xi)}|\xi|^{\alpha}d\xi,\quad (b\in \N,\alpha\in\R),
 	\end{aligned}	
 \end{equation}
with $b=3$ are crucial in obtaining such estimates, see \cite{Kenig_gKdV_1993,Linares_book}. Although there are various estimates on the sizes of this type of integrals (see \cite{Kenig_gKdV_1993,Linares_book}), we carry out a more extensive investigation and obtain sharp estimates  for  $I^{b,\alpha}$ with $b\geq 2$, see Section \ref{section:est_on_tail}. This may be of independent interest and may find applications to equations with different dispersion relations. The same methods also yield sharp estimates for the operator
	\begin{equation}\label{osc11}
	J^{b,\alpha}(x):=\begin{aligned}
		\int_{\R} e^{i(|\xi|^b+x\xi)}|\xi|^{\alpha}d\xi,\quad (b\in \R,\alpha\in\R),
	\end{aligned}	
\end{equation}
for $b>1$. These fractional power estimates may prove useful in the study of fractional PDEs.
\end{Remark}
\begin{Remark}
    For the proof of Theorems \ref{thm:scat_in_L2} and \ref{thm:scat_in_H1}, we impose a mild decay assumption $|g(t)| = o(t^{-\gamma})$, $\omega$-a.s., with $\gamma > \frac{2}{3}$, which is due to the analysis carried out in Section \ref{section:est_on_tail} (see the proof of Proposition \ref{mainprop}). This ensures that the tail of the stochastic convolution has bounded space-time norms required by the scattering theory. The global in time analysis of the tail of the stochastic convolution $z_{*}(t)$ is based on refined oscillatory integral estimates for $I^{b,\alpha}(x)$ defined as in \eqref{osc1}, covering the full range $-1 < \alpha < b-1$, $b>1$, and their time scaled versions (see Section \ref{section:est_on_tail}). Combining these bounds with deterministic scattering results (see \cite{Dodson_gKdV} for the mass-critical equation and Proposition \ref{propfarah} for the mass-supercritical equation) allows us to construct the scattering state and conclude almost sure scattering.
\end{Remark}
The organization of the paper is as follows. In Section \ref{section: preliminaries}, we give notations and some useful lemmas that will be used frequently in the rest of the paper. In Section \ref{section:well-posedness}, we prove Theorem \ref{thm:well-posedness}. Section \ref{section:est_on_tail} is devoted to the estimation of the tail of the stochastic convolution, including new and refined oscillatory integral estimates pertaining to more general dispersion relations. Lastly, in Section \ref{section:scat_in_L2}, we prove Theorems \ref{thm:scat_in_L2} and \ref{thm:scat_in_H1}, respectively.

\section{Notations, function spaces and useful lemmas}\label{section: preliminaries}
Throughout, we use $c, C$ to denote various constants that may change line by line. In addition, the notation $C_{\alpha_1,...,\alpha_n}$ is used for a constant depending on $\alpha_j$, $j=1,...,n$, ($\alpha_j$ may be a certain exponent, an outcome for a probability space etc); alternatively, $C(\beta_1,..., \beta_n)$ (likewise $c(\beta_1,..., \beta_n)$) is defined in the same way (where $\beta_j$ may be a parameter, a function, etc.). For nonnegative quantities $X,Y$, we denote $X\lesssim Y$ if there is a constant $C>0$ independent of $X,Y$ such that $X\leq CY$ and write $X\lesssim_{a,b}Y$ when we have $X\leq C(a,b)Y$ for $a,b$ being certain parameters. We denote $X\sim Y$ if both $X\lesssim Y$ and $Y\lesssim X$ hold. The notation $X\ll Y$ is used when $X<CY$ for a sufficiently small constant $C>0$, depending on the context. Also, the notation $\langle \cdot \rangle$ is used to denote $\langle X\rangle=(1+\vert X\vert^{2})^{\frac{1}{2}}$. We also use $\alpha\pm$ to denote $\alpha\pm\varepsilon$ for $0<\varepsilon\ll 1$.

The mixed Lebesgue spaces $L^q_t L^p_x (I \times \mathbb{R})$ are defined by
\begin{equation*}
    L^{q}_{t}L^{p}_{x}(I\times\mathbb{R}):=\{f:I\times\mathbb{R}\to\mathbb{C}\mid \Vert f\Vert_{L^{q}_{t}L^{p}_{x}(I\times\mathbb{R})}<\infty\}
\end{equation*}
where $p,q\geq 1$, $I$ is an interval in $\mathbb{R}$, and
\begin{equation*}
    \Vert f\Vert_{L^{q}_{t}L^{p}_{x}(I\times\mathbb{R})}:=\left(\int_{I}\left(\int_{\mathbb{R}}\vert f(t,x)\vert^{p}dx\right)^{\frac{q}{p}}dt\right)^{\frac{1}{q}}.
\end{equation*}
In addition, a similar definition applies to $L^{r}_{\omega}L^{q}_{t}L^{p}_{x}(\Omega\times I\times\R)$. For convenience, unless otherwise stated, we usually omit the domains of the variables $x$ and $\omega$ in mixed norms: $$\norm{\cdot}_{L^{p}_{x}L^{q}_{t}L^{r}_{\omega}([0,T])}:=\norm{\cdot}_{L^{p}_{x}L^{q}_{t}L^{r}_{\omega}(\mathbb{R}\times[0,T]\times\Omega)}.$$ The Fourier transform $\widehat{f}$ of a function $f$ is defined by 
\begin{equation*}
    \widehat{f}(\xi)=\int_{\mathbb{R}}e^{-ix\xi}f(x)dx.
\end{equation*}
Let $1<p<\infty$ and $s\in\mathbb{R}$. Then, the Sobolev spaces $W^{s,p}_{x}(\mathbb{R})$ are defined as the closure of the Schwartz class functions under the norm
\begin{equation*}
    \Vert f\Vert_{W^{s,p}_{x}}:=\Vert\langle D_x \rangle^{s}f\Vert_{L^{p}_{x}}
\end{equation*}
where the Fourier multiplier operator $\langle D^{s}_{x} \rangle$ is defined as 
\begin{equation*}
    \widehat{\langle D^{s}_{x} \rangle f}(\xi)=\langle \xi\rangle ^{s}\widehat{f}(\xi),\quad s\in\R.
\end{equation*}
In the particular case $p=2$, due to Plancherel identity, the space $H^{s}_{x}=W^{s,2}_{x}$ is defined by 
\begin{equation*}
    \Vert f\Vert_{H^{s}_{x}}:=\Vert\langle\xi\rangle^{s}\widehat{f}\Vert_{L^{2}_{\xi}}.
\end{equation*}
The Sobolev embedding $W^{s,p}_{x}(\mathbb{R})\hookrightarrow L^{q}_{x}(\mathbb{R})$ holds for $1<p<q<\infty$, $s>0$, and $\frac{1}{q}\geq\frac{1}{p}-s$. In what follows, we list the useful lemmas that we will use throughout our discussion.
\begin{lemma}[Van der Corput, \cite{Linares_book}]\label{lemma:van der Corput}
Let $k\in \mathbb{N}$. Assume that $\vert\phi^{(k)}(x)\vert\geq \sigma>0$ for any $x\in [a,b]$ and $\phi'(x)$ is monotonic if $k=1$. Then, we have  
\begin{equation*}
    \left\vert\int_{a}^{b}e^{i\lambda\phi(x)}f(x)dx\right\vert\leq c_{k}(\lambda\sigma)^{-\frac{1}{k}}(\norm{f}_{L^{\infty}}+\norm{f'}_{L^{1}})
\end{equation*}
where $c_{k}$ is independent of $a,b$.
\end{lemma}
\begin{lemma}[Burkholder-Davis-Gundy inequality, Hypothesis 6.4 in \cite{Da_Prato_book}] \label{lemma:B-D-G}
Let $X$ be a Banach space and $\mathcal{N}_{W}^{p}(0,T;X)$, $p\geq 1$, be the space of all $X$-valued predictable processes $\phi$ such that
\begin{equation*}
    \norm{\phi}_{\mathcal{N}_{W}^{p}(0,T;X)}^{p}=\int_{0}^{T}\mathbb{E}\left[\norm{\phi(s)}_{X}^{p}\right] ds<\infty
\end{equation*}
where $W(t)$ is a Brownian motion in $(\Omega,\mathcal{F},\mathbb{P})$ adapted to the filtration $(\mathcal{F}_{t})_{t\geq 0}$. There exists $C_{p}$ such that for all $\phi\in \mathcal{N}_{W}^{p}(0,T;X)$, we have
\begin{equation}\label{B-D-G_inequality}
    \mathbb{E}\left[\norm{\int_{0}^{T}\phi(s)dW(s)}_{X}^{p}\right]\leq C_{p}\mathbb{E}\left[\left(\int_{0}^{T}\norm{\phi(s)}_{X}^{2}ds\right)^{p/2}\right].
\end{equation}
\end{lemma}
\begin{lemma}[Stochastic Gronwall's Inequality, \cite{Xicheng_Zhang_Gronwall}]\label{gronwall}
    Let $\xi(t)$ and $\eta(t)$ be two nonnegative cadlag $\mathcal{F}_t$-adapted process, $A_t$ a continuous nondecreasing $\mathcal{F}_t$-adapted processes with $A_0 = 0$ and $M_t$ a local martingale with $M_0 = 0$. Suppose that
\begin{equation*}
    \xi(t) \leq \eta(t) + \int_0^t \xi(s) dA_s + M_t
\end{equation*}
for any $t \geq 0$. Then for any $0 < q < p < 1$ and $\tau > 0$, we have
\begin{equation*}
    \left( \mathbb{E}\left[\sup\limits_{0 \leq t \leq \tau} \xi(t)^q \right] \right)^{1/q} \leq \left( \frac{p}{p-q} \right)^{1/q} \big( \mathbb{E}\left[\exp\left( p A_{\tau} / (1-p) \right) \right] \big)^{(1-p)/p} \mathbb{E}\left[\sup\limits_{0 \leq t \leq \tau} \eta(t) \right].
\end{equation*}
\end{lemma}
In addition, we need the following lemma \cite[Lemma 2.17]{saanouni2015remarks}:
\begin{lemma}\label{lemma:uniform_bounded_by_its_higher_powers}
Let $T>0$ and $f\in C([0,T];\mathbb{R}_{+})$, such that
\begin{equation*}
    f(t)\leq a+b f(t)^{\alpha},\quad \text{for $t\in[0,T]$,}
\end{equation*}
where $a,b>0$, $\alpha>1$, $a<(1-\frac{1}{\alpha})(\alpha b)^{-\frac{1}{\alpha-1}}$, and $f(0)\leq (\alpha b)^{-\frac{1}{\alpha-1}}$. Then
\begin{equation*}
    f(t)\leq \frac{\alpha}{\alpha-1}a,\quad\text{for all $t\in[0,T]$}.
\end{equation*}
\end{lemma}
Next, we recall the linear Strichartz estimates for the Airy propagator and the inhomogeneous initial value problem:
\begin{equation}\label{eq:inhom_airy}
    \begin{cases}
        &u_{t}+u_{xxx}=g(t,x),\\
        &u(0,x)=u_{0}(x).
    \end{cases}
     (t,x)\in\mathbb{R}\times\mathbb{R},
\end{equation}
\begin{lemma}[Strichartz estimates, \cite{ERDOGAN_PDE_Book}]\label{lemma:strichartz}
    Let $\beta\in[0,\frac{1}{2}]$ and the pair $(p,q)$ satisfy the following admissibility condition 
    \begin{equation}\label{eq:admissible_pair}
    \frac{1}{q}=\frac{(\beta+1)}{3}\left(\frac{1}{2}-\frac{1}{p}\right)
    \end{equation}
    for $2\leq p\leq \infty$. Then, we have
    \begin{align*}
        &\Vert D^{\frac{\beta}{2}(\frac{1}{p'}-\frac{1}{p})}V(t)u_{0}\Vert_{L^{q}_{t}L^{p}_{x}}\lesssim\Vert u_{0}\Vert_{L^{2}_{x}},
        \\
        &\Big\Vert D^{\frac{\beta}{2}(\frac{1}{p'}-\frac{1}{p})}\int_{\mathbb{R}}V(-t)g(t,\cdot)dt\Big\Vert_{L^{2}_{x}}\lesssim\Vert g\Vert_{L^{q'}_{t}L^{p'}_{x}},
        \\
        &\Big\Vert D^{\beta(\frac{1}{p'}-\frac{1}{p})}\int_{\mathbb{R}}V(t-t')g(t',\cdot)dt'\Big\Vert_{L^{q}_{t}L^{p}_{x}}\lesssim\Vert g\Vert_{L^{q'}_{t}L^{p'}_{x}}.
    \end{align*}
\end{lemma}
Now, we state the Kato-type estimates (see Section 3 of \cite{Kenig_gKdV_1993}).
\begin{lemma}[Kato Estimates]\label{kato_estimates}
    Let the triple $(p,q,\alpha)$ satisfy the (Kato) admissibility condition
    \begin{equation}\label{eq:Kato_admissible}
    \frac{2}{p}=\frac{1}{2}-\frac{1}{q}\quad\text{and}\quad \alpha=\frac{2}{q}-\frac{1}{p},\quad (p,q,\alpha)\in[4,\infty]\times[2,\infty]\times[-\frac{1}{4},1].
    \end{equation}
    Then
    \begin{align}
    &\Vert D_{x}^{\alpha}V(t)u_{0}\Vert_{L^{p}_{x}L^{q}_{t}}\lesssim\Vert u_{0}\Vert_{L^{2}_{x}},\label{est:Kato_1}\\
    &\Vert D_{x}^{\alpha}\int_{\mathbb{R}}V(-t')g(t',\cdot)dt'\Vert_{L^{2}_{x}}\lesssim\Vert g\Vert_{L^{p'}_{x}L^{q'}_{t}},\label{est:Kato_2}\\
    &\Vert D_{x}^{2\alpha}\int_{0}^{t}V(t-t')g(t',\cdot)dt'\Vert_{L^{p}_{x}L^{q}_{t}}\lesssim\Vert g\Vert_{L^{p'}_{x}L^{q'}_{t}}.\label{est:Kato_3}
    \end{align}
\end{lemma}
\begin{proposition}[\cite{Dodson_gKdV}]\label{Kato_est_most_general}
    Let $u$ be the solution of the inhomogeneous initial value problem \eqref{eq:inhom_airy} for $t\in I$ where $I$ is an interval with $0\in I$. Then for any (Kato) admissible triples $(p_{1},q_{1},\alpha_{1})$ and $(p_{2},q_{2},\alpha_{2})$, satisfying \eqref{eq:Kato_admissible}, we have
    \begin{equation}\label{Kato_estimate_Dodson}
        \Vert D_{x}^{\alpha_{1}}u\Vert_{L^{p_{1}}_{x}L^{q_{1}}_{t}(I)}\lesssim\Vert u_{0}\Vert_{L^{2}_{x}}+\Vert D^{-\alpha_{2}}_{x}g\Vert_{L^{p_{2}'}_{x}L^{q_{2}'}_{t}(I)}.
    \end{equation}
\end{proposition}
\begin{lemma}[Lemma 2.9 in \cite{Farah_H1_scattering}]\label{lemma:fractional_leibniz}
    Let $0<\alpha<1$ and $p,p_1,p_2,q,q_1,q_2\in(1,\infty)$ with $\frac{1}{p}=\frac{1}{p_1}+\frac{1}{p_2}$ and $\frac{1}{p}=\frac{1}{p_1}+\frac{1}{p_2}$. Then
    \begin{enumerate}[(i)]
        \item $\norm{D_{x}^{\alpha}(fg)-fD^{\alpha}_{x}g-gD^{\alpha}_{x} f}_{L^{p}_x L^{q}_t}\lesssim\norm{D^{\alpha}_x f}_{L^{p_1}_x L^{q_1}_t}\norm{g}_{L^{p_2}_x L^{q_2}_t}.$\label{fractional_leibniz_1}

        \noindent
        The same still holds if $p=1$ and $q=2$.
        \item $\norm{D^{\alpha}_x F(f)}_{L^{p}_x L^{q}_t}\lesssim \norm{D^{\alpha}_x f}_{L^{p_1}_x L^{q_1}_t}\norm{F'(f)}_{L^{p_2}_x L^{q_2}_t}.$ \label{fractional_leibniz_2}
    \end{enumerate}
\end{lemma}
\section{Proof of Theorem \ref{thm:well-posedness}}\label{section:well-posedness}
In this section, we establish Theorem \ref{thm:well-posedness} in two separate subsections: we first establish the local existence theory by proving some required a priori estimates; and later in this section, we upgrade it to the global existence of the solutions. 
\subsection{Local well-posedness} \label{locwelres}
In the following section, we establish a priori estimates, which will be utilized later in this section in order to conclude the local existence theory.
\subsubsection{A priori estimates}\label{A priori estimates}
In this section, we intend to show that the norms defined in the following are finite for $0<T<\infty$:
\begin{equation}\label{spacetime_norms}    
    \begin{aligned}
    \beta_{1,T}(z) &= \mathbb{E} \left[ \sup\limits_{0 \leq t \leq T}\Vert D_{x}^{s_{k}} z(t) \Vert_{L^2_x}^{2} \right]=:\mathbb{E}[\alpha_{1,T}^2(z)], \\
    \beta_{2,T}(z) &= \mathbb{E} \left[ \Vert D_{x}^{1+s_{k}}z \Vert_{L^{\infty}_x L^2_t([0,T])}^{2} \right]=:\mathbb{E}[\alpha_{2,T}^2(z)], \\
    \beta_{3,T}(z) &= \mathbb{E} \left[ \Vert D_{x}^{s_{k}} z \Vert_{L^5_x L^{10}_t([0,T])}^2 \right]=:\mathbb{E}[\alpha_{3,T}^2(z)], \\ 
    \beta_{4,T}(z) &= \mathbb{E} \left[ \Vert D_x^{\frac{1}{10} - \frac{2}{5k}} D_t^{\frac{3}{10}-\frac{6}{5k}} z \Vert_{L^{p_k}_x L^{q_k}_t([0,T])}^2 \right]=:\mathbb{E}[\alpha_{4,T}^2(z)],
    \end{aligned}
\end{equation}    
where $s_k$ is as in \eqref{scalingcr}, and \begin{align}\label{pkqk}
 \frac{1}{p_k}=\frac{2}{5k}+\frac{1}{10}\quad\text{and}\quad\frac{1}{q_k}=\frac{3}{10}-\frac{4}{5k}\,\,\,\text{for}\,\,k\geq 4.  
\end{align} 
In fact, the quantities in \eqref{spacetime_norms} will be controlled by various norms of $\phi$ and $g$ (which we always denote by $C(\phi,g)$) multiplied by some powers of $T$. This will eventually imply an $\omega$ almost sure boundedness of $\alpha_{j,T}(z)$, $j\in\{1,...,4\}$. Notice that in the case $k=4$, we have $\beta_{3,T}= \beta_{4,T}$.

We begin by estimating $\beta_{1,T}$. Thus, using the $L^{2}_{x}$-isometry property of $V(t)$ and Burkholder-Davis-Gundy inequality \eqref{B-D-G_inequality}, we obtain
\begin{equation}\label{beta1T}
\begin{aligned}
    \beta_{1,T}(z)&\lesssim\mathbb{E}\left[\int_{0}^{T}\norm{V(-t')D_{x}^{s_k}\phi g(t')}_{L^{2}_{x}}^{2}dt'\right]\\
    \lesssim&\norm{D_{x}^{s_k}\phi}_{L^{2}_{x}}^2\mathbb{E}\left[\norm{g}_{L^{2}_{t}([0,T])}^2\right]\\
    \lesssim&\norm{D_{x}^{s_k}\phi}_{L^{2}_{x}}^2\norm{g}_{L^{\infty}_{\omega,t}([0,T])}^2 T\leq C(\phi,g)T.
\end{aligned}
\end{equation}

With regard to $\beta_{2,T}$, we will implement an interpolation argument as in the proof of \cite[Lemma 2.4]{Millet_SKdV}, and assume that $2\leq q\leq p<\infty$ and $\alpha\geq 0$. Using Hölder inequality in the $t$-variable, the Burkholder-Davis-Gundy inequality \eqref{B-D-G_inequality}, and the Fourier inversion formula, we obtain the bound:
\begin{equation}\label{dalphaz}
\begin{aligned}
    \norm{D_{x}^{\alpha}z}_{L^{\infty}_{x}L^{p}_{\omega}L^{q}_{t}([0,T])}^{p}\leq&2 T^{\frac{p}{q}-1}\sup\limits_{x\in\mathbb{R}}\mathbb{E}\Big[\int_{0}^{T}\Big\vert\int_{0}^{t}V(t-t')D_{x}^{\alpha}\phi(x)g(t')dB(t')\Big\vert^{p}dt\Big]\\
    \lesssim&T^{\frac{p}{q}-1}\sup\limits_{x\in\mathbb{R}}\mathbb{E}\Big[\int_{0}^{T}\Big(\int_{0}^{t}\vert V(t-t')D_{x}^{\alpha}\phi(x)\vert^{2}g^{2}(t')dt'\Big)^{p/2}dt\Big]\\
    \lesssim&T^{\frac{p}{q}}\sup\limits_{x\in\mathbb{R}}\mathbb{E}\Big[\Big(\int_{0}^{T}g^{2}(t')\Big\vert\int_{\mathbb{R}}e^{ix\xi-i(t-t')\xi^{3}}\vert\xi\vert^{\alpha}\widehat{\phi}(\xi)d\xi\Big\vert^{2} dt'\Big)^{p/2}\Big]\\
    \lesssim&T^{\frac{p}{q}}\mathbb{E}\left[\norm{g}_{L^{2}_{t}([0,T])}^{p}\right]\norm{\phi}_{H^{\frac{1}{2}+\alpha+}_{x}}\\
    \lesssim& T^{\frac{p}{q}+\frac{p}{2}}\norm{\phi}_{H^{\frac{1}{2}+\alpha+}_{x}}\norm{g}_{L^{\infty}_{\omega,t}([0,T])}^{p}\leq C(\phi,g)T^{\frac{p}{q}+\frac{p}{2}},
\end{aligned}
\end{equation}
which yields 
\begin{equation}\label{first_interpolation_norm_beta_2}
\norm{D^{\alpha}_{x}z}_{L^{\infty}_{x}L^{p}_{\omega}L^{2}_{t}([0,T])}\leq C(\phi,g) T  \quad \text{for}\,\, p\geq 2.
\end{equation} 
We also need to estimate the size of $\norm{z}_{L^{2}_{x}L^{p}_{\omega}L^{2}_{t}([0,T])}$. Therefore, application of the Hölder inequality in $t$, the Burkholder-Davis-Gundy inequality \eqref{B-D-G_inequality}, and Young's convolution inequality gives rise to the following:
\begin{equation}\label{process_1}
\begin{aligned}
    \norm{z}_{L^{2}_{x}L^{p}_{\omega}L^{2}_{t}([0,T])}^{2}&
    \leq T^{1-\frac{2}{p}}\int_{\mathbb{R}}\Big(\mathbb{E}\Big[\int_{0}^{T}\Big\vert\int_{0}^{t}V(t-t')\phi(x)g(t')dB(t')\Big\vert^{p} dt\Big]\Big)^{2/p}dx\\
    \lesssim&T^{1-\frac{2}{p}}\int_{\mathbb{R}}\Big(\mathbb{E}\Big[\int_{0}^{T}\Big(\int_{0}^{t}\vert V(t-t')\phi(x)g(t')\vert^{2}dt'\Big)^{p/2} dt\Big]\Big)^{2/p}dx\\
    \lesssim& T\int_{\mathbb{R}}\Big(\mathbb{E}\Big[\Big(\int_{0}^{T}\vert V(t-t')\phi(x)g(t')\vert^{2}dt'\Big)^{p/2}\Big]\Big)^{2/p}dx\\
    \lesssim&T\int_{\mathbb{R}}\Big(\mathbb{E}\Big[\norm{\vert V(\cdot)\phi(x)\vert^{2}*g^{2}(\cdot)}_{L^{\infty}_{t}([0,T])}^{p/2}\Big]\Big)^{2/p}dx\\
    \lesssim& T\,\mathbb{E}\left[\norm{g}_{L^{2}_{t}([0,T])}^{p}\right]^{\frac{2}{p}}\int_{\mathbb{R}}\sup\limits_{t\in[0,T]}\vert V(t)\phi(x)\vert^{2} dx\\
    \lesssim& T^{2}\norm{g}_{L^{\infty}_{\omega,t}([0,T])}\int_{\mathbb{R}}\sup\limits_{t\in[0,T]}\vert V(t)\phi(x)\vert^{2} dx.
\end{aligned}
\end{equation}
The Sobolev embedding in time and the identity $\partial_{t}V(t)\phi(x)=V(t)\phi_{xxx}(x)$ lead to
\begin{equation*}
\norm{V(t)\phi(x)}_{L^{2}_{x}L^{\infty}_{t}([0,T])}^{2}\lesssim\norm{V(t)(1+\partial_{x}^{3})\phi}_{L^{2}_{t}L^{2}_{x}([0,T])}^{2}\leq C(\phi)T.
\end{equation*}
Putting the above inequality in \eqref{process_1} gives that
\begin{equation}\label{second_interpolation_norm_beta_2}
    \norm{z}_{L^{2}_{x}L^{p}_{\omega}L^{2}_{t}([0,T])}\leq C(\phi,g) T^{\frac{3}{2}}.
\end{equation}
Now choosing $\alpha>1+s_k$ and interpolating \eqref{first_interpolation_norm_beta_2} and \eqref{second_interpolation_norm_beta_2} (see \cite[Proposition A1]{deBouard_SKdV_1}), we obtain
\begin{equation*}
    \norm{D^{\alpha\theta}_{x}z}_{L^{p}_{\omega}L^{p}_{x}L^{2}_{t}([0,T])}\leq C(\phi,g)T^{\frac{3}{2}-\frac{\theta}{2}}
\end{equation*}
with $\theta=1-\frac{2}{p}$. By embedding $W^{\frac{1}{p}+,p}_{x}\hookrightarrow L^{\infty}_{x}$, we thus have (with $\theta=1-\frac{2}{p}$)
\begin{equation}\label{dalpfha}
    \norm{D^{\alpha\theta-\frac{1}{p}-}z}_{L^{p}_{\omega}L^{\infty}_{x}L^{2}_{t}([0,T])}\lesssim \norm{D^{\alpha\theta}_{x}z}_{L^{p}_{\omega}L^{p}_{x}L^{2}_{t}([0,T])}\lesssim C(\phi,g)T^{1+\frac{1}{p}}.
\end{equation}
We require $\alpha\theta-\frac{1}{p}-=1+s_{k}$, which is equivalent to having $\alpha=\frac{p(3k-4)+2k}{2k(p-2)}+$. Therefore, for $2<p<\infty$ and $\alpha=\frac{p(3k-4)+2k}{2k(p-2)}+$, the estimate \eqref{dalpfha} implies that
\begin{equation*}
    \beta_{2,T}(z)\leq\norm{D^{1+s_k}_{x}z}_{L^{p}_{\omega}L^{\infty}_{x}L^{2}_{t}([0,T])}^2\leq C(\phi,g)T^{2+\frac{2}{p}}.
\end{equation*}
Thus, setting $p=\infty-$, we get
\begin{equation}\label{beta2T}
    \beta_{2,T}(z)\leq C(\phi,g)T^{2+}.
\end{equation}

Next, we deal with $\beta_{3,T}$. Using Hölder inequality in $\omega$, Burkholder-Davis-Gundy inequality \eqref{B-D-G_inequality}, and Young's convolution inequality, we obtain (with $2\leq p\leq q<\infty$ and $\alpha\in\mathbb{R}$) the following:
\begin{equation}\label{process_2}
\begin{aligned}
    \mathbb{E}\left[\norm{D^{\alpha}_{x}z}^{p}_{L^{p}_{x}L^{q}_{t}([0,T])}\right]\leq&\norm{\left(\mathbb{E}\left[\Big\vert\int_{0}^{t}V(t-t')D^{\alpha}_{x}\phi(x)g(t')dB(t')\Big\vert^{q}\right]\right)^{1/q}}_{L^{p}_{x}L^{q}_{t}([0,T])}^{p}\\
    \lesssim&\int_{\mathbb{R}}\left(\int_{0}^{T}\mathbb{E}\left[\left(\int_{0}^{t}\vert V(t-t')D^{\alpha}_{x}\phi(x)g(t')\vert^{2}dt'\right)^{q/2}\right] dt\right)^{p/q}dx\\
    \lesssim&\int_{\mathbb{R}}\left(\mathbb{E}\left[\int_{0}^{T}\left(\vert V(\cdot)D^{\alpha}_{x}\phi(x)\vert^{2}*\vert g(\cdot)\vert^{2}\right)^{q/2}(t)dt\right]\right)^{p/q}dx\\
    \lesssim&\int_{\mathbb{R}}\left(\mathbb{E}\left[\norm{g}_{L^{2}_{t}([0,T])}^{q}\norm{D^{\alpha}_{x}V(t)\phi(x)}_{L^{q}_{t}([0,T])}^{q}\right]\right)^{p/q}dx\\
    \leq& T^{\frac{p}{2}}C(g)\norm{D^{\alpha}_{x}V(t)\phi(x)}_{L^{p}_{x}L^{q}_{t}([0,T])}^{p}.
\end{aligned}
\end{equation}
Setting $(p,q,\alpha)=(5,10,s_k)$ in \eqref{process_2}, and then applying the Kato estimate \eqref{est:Kato_1}, we get to the bound
\begin{equation*}
    \mathbb{E}\left[\norm{D^{s_k}_{x}z}^{5}_{L^{5}_{x}L^{10}_{t}([0,T])}\right]\leq T^{\frac{5}{2}} C(g)\norm{D_{x}^{s_k}\phi}_{L^{2}_{x}}^{5},
\end{equation*}
which yields the estimate 
\begin{align}\label{beta3T}
    \beta_{3,T}(z)\leq C(\phi,g)T.
\end{align}

Lastly, for $\beta_{4,T}$, we first note that $p_k\leq q_k$ if and only if $k\leq 6$ (where $p_k$ and $q_k$ are defined as in \eqref{pkqk}). Hence, when $k\leq 6$, the estimate \eqref{process_2} implies at once that
\begin{equation}\label{first_bound_for_beta_4}
    \beta_{4,T}(z)\leq C(\phi,g)T.
\end{equation}
In the case $k>6$, let $\alpha\in\mathbb{R}^+$, which is to be chosen later. We will make use of the interpolation argument, for this purpose, first recalling from \eqref{dalphaz} (as $k>6$, in view of the remark above, we consider $q<p$, which is in line with the assumption on $p$, $q$ made before \eqref{dalphaz}) that
\begin{equation}\label{interpolation_1}
    \norm{D^{\alpha}_{x}z}_{L^{\infty}_{x}L^{p}_{\omega}L^{q}_{t}([0,T])}\leq C(\phi,g)T^{\frac{1}{q}+\frac{1}{2}}.
\end{equation}
A similar calculation as in \eqref{process_1} can be performed (with the time integrability exponent $q\geq 2$ instead) to obtain
\begin{equation}\label{process_3}
\norm{z}_{L^{2}_{x}L^{p}_{\omega}L^{q}_{t}([0,T])}^{2}
    \leq T^{\frac{2}{q}+1}C(g)\int_{\mathbb{R}}\sup\limits_{0\leq t\leq T}\vert V(t)\phi(x)\vert^{2}dx.
\end{equation}
Thus, using \eqref{process_3}, the Sobolev embedding in $t$, $H^{1}_{t}\hookrightarrow H^{\frac{1}{2}+}_{t}\hookrightarrow L^{\infty}_{t}$, and $\partial_{t}V(t)\phi(x)\sim V(t)\partial_{x}^{3}\phi(x)$, we have the following
\begin{equation}\label{interpolation_2}
    \begin{aligned}
\norm{z}_{L^{2}_{x}L^{p}_{\omega}L^{q}_{t}([0,T])}\leq& C(g)T^{\frac{1}{q}+\frac{1}{2}}\norm{V(t)\phi(x)}_{L^{2}_{t}H^{3}_{x}([0,T])}\\
    \leq&C(\phi,g)T^{\frac{1}{q}+1}.
\end{aligned}
\end{equation}
As a result, interpolating \eqref{interpolation_1} and \eqref{interpolation_2} (by a similar argument as in \cite[Proposition A1]{deBouard_SKdV_1} with the parameter $\theta=1-\frac{2}{p}$, we arrive at
\begin{equation*}
   \norm{D^{\alpha(1-\frac{2}{p})}_{x}z}_{L^{p}_{\omega,x}L^{q}_{t}([0,T])}\leq C(\phi,g)T^{\frac{1}{2}+\frac{1}{q}+\frac{1}{p}}. 
\end{equation*}
Finally, choosing $p=p_k$ and $q=q_k$, and $\alpha=\frac{5(k-4)}{4(k-1)}$, we achieve the following:
\begin{equation}\label{second_bound_for_beta_4}
    \beta_{4,T}(z)\leq C(\phi,g)T^{\frac{9}{5}-\frac{4}{5k}}.
\end{equation}
Combining \eqref{first_bound_for_beta_4} and \eqref{second_bound_for_beta_4}, for $k\geq 4$, we have
\begin{equation}\label{boundforbeta4}
    \beta_{4,T}(z)\leq C(\phi,g)(T^{\frac{9}{5}-\frac{4}{5k}}+T).
\end{equation}
Therefore, from \eqref{beta1T}, \eqref{beta2T}, \eqref{beta3T}, and \eqref{boundforbeta4}, we obtain
\begin{equation}\label{bound_for_z(t)_with_expectation}
    \max_{1\leq j \leq 4}\beta_{j,T}(z)\leq C(\phi,g)(T+T^{2+}),
\end{equation}
which further shows (in view of Doob's martingale convergence theorems) that
\begin{equation}\label{almost_sure_bound_stochastic_convolution}
    \max_{1\leq j \leq 4}\alpha_{j,T}(z)<\infty\quad\text{$\omega$-a.s. for $0<T<\infty$.}
\end{equation}

\subsubsection{Fixed point argument} 
In this part, we will show the fixed point argument only in the critical regularity for simplicity, with minor modifications a similar argument follows in the case $k>4$ as well. To begin with, in the separate cases $k=4$ and $k>4$, we give precise bounds to the terms $\alpha_{j,T}(u)$, $j=1,2,3,4$, where $\alpha_{j,T}(\cdot)$ is defined in \eqref{spacetime_norms}, and $u$ satisfies the Duhamel formula \eqref{Duhamel_of_stochastic_solution}. While controlling $\alpha_{j,T}(u)$ norms, \eqref{almost_sure_bound_stochastic_convolution} provides us with pathwise finiteness for the stochastic convolution in the Duhamel formula \eqref{Duhamel_of_stochastic_solution}. We start with $\alpha_{1,T}(u)$. By Lemma \ref{lemma:strichartz}, Lemma \ref{kato_estimates}, Sobolev embedding (in $x$ and $t$), and \eqref{fractional_leibniz_2} of Lemma \ref{lemma:fractional_leibniz} together with the Hölder inequality applied for the cases $k>4$ and $k=4$ respectively, we obtain the bound:
\begin{equation}\label{alpha1u}
\begin{aligned}
   \alpha_{1,T}(u)\lesssim&\norm{D_{x}^{s_k}u_0}_{L^{2}_x}+\norm{D_{x}^{s_k}(u^{k+1})}_{L^{1}_{x}L^{2}_{t}([0,T])}+\alpha_{1,T}(z)\\
    \lesssim&\begin{cases}
              \norm{u_0}_{L^{2}_x}+\norm{u}_{L^{5}_{x}L^{10}_{t}([0,T])}^{5}+\alpha_{1,T}(z), &\text{if}\,\,k=4,\\
              \norm{D_{x}^{s_k}u_0}_{L^{2}_x}+\alpha_{3,T}(u)(\alpha_{4,T}(u))^{k}+\alpha_{1,T}(z),  &\text{if}\,\,k>4.
         \end{cases}
\end{aligned}
\end{equation}
In a similar way, we have
\begin{equation}\label{alpha2u}
\begin{aligned}
    \alpha_{2,T}(u)\lesssim&\norm{D_{x}^{s_k}u_{0}}_{L^{2}_{x}}+\norm{D_{x}^{s_k}(u^{k+1})}_{L^{1}_{x}L^{2}_{t}([0,T])}+\alpha_{2,T}(z)\\
    \lesssim&\begin{cases}
            \norm{u_{0}}_{L^{2}_{x}}+\norm{u}_{L^{5}_{x}L^{10}_{t}([0,T])}^{5}+\alpha_{2,T}(z),  &\text{if}\,\,k=4,\\
            \norm{D_{x}^{s_k}u_{0}}_{L^{2}_{x}}+\norm{u_{x}}_{L^{5}_{x}L^{10}_{t}([0,T])}(\alpha_{4,T}(u))^{k}+\alpha_{2,T}(z),  &\text{if}\,\,k>4,
         \end{cases}
\end{aligned}
\end{equation}
and
\begin{equation}\label{alpha3u}
\begin{aligned}
   \alpha_{3,T}(u)\lesssim&
    \begin{cases}
    \norm{u_{0}}_{L^{2}_{x}}+\norm{u}_{L^{5}_{x}L^{10}_{t}([0,T])}^{5}+\alpha_{3,T}(z), &\text{if}\,\,k=4,\\
    \norm{D^{s_k}_{x}u_{0}}_{L^{2}_{x}}+\alpha_{3,T}(u)(\alpha_{4,T}(u))^{k}+\alpha_{3,T}(z),&\text{if}\,\,k>4.
    \end{cases}
\end{aligned}
\end{equation}
Recall that when $k=4$, we have $\alpha_{3,T}=\alpha_{4,T}$, therefore, it suffices, for $\alpha_{4,T}(u)$, to consider only the case $k>4$. In this respect, in light of the proof of \cite[Proposition 6.1]{Kenig_gKdV_1993}, we get
\begin{equation}\label{1,}
\begin{aligned}
\alpha_{4,T}(u)\lesssim&\norm{D_{x}^{s_k}u_{0}}_{L^{2}_{x}}+\norm{u_{x}u^{k}}_{L^{p'_{k}}_{x}L^{q'_k}_{t}([0,T])}+ \alpha_{4,T}(z)  \\
    \lesssim&\norm{D_{x}^{s_k}u_{0}}_{L^{2}_{x}}+\norm{u_{x}}_{L^{r}_{x}L^{l}_{t}([0,T])}\norm{u}_{L^{\frac{5k}{4}}_{x}L^{\frac{5k}{2}}_{t}([0,T])}^{k}+\alpha_{4,T}(z),
\end{aligned}
\end{equation}
where $\frac{1}{r}=\frac{1}{10}-\frac{2}{5k}$ and $\frac{1}{l}=\frac{3}{10}+\frac{4}{5k}$, $k>4$. To the middle term in \eqref{1,} applying Stein's complex interpolation (with the interpolation parameter $\theta=\frac{1}{2}+\frac{2}{k}$) and the Sobolev embedding $\dot{W}_x^{p_k,\frac{1}{10}-\frac{2}{5k}} \dot{W}_t^{q_k,\frac{3}{10}-\frac{6}{5k}}\hookrightarrow L^{\frac{5k}{4}}_{x}L^{\frac{5k}{2}}_{t}$, for $p_k$, $q_k$ as in \eqref{pkqk}, leads to the final form of the bound:
\begin{equation}\label{alpha4u}
\alpha_{4,T}(u) \lesssim \norm{D_{x}^{s_k}u_{0}}_{L^{2}_{x}}+(\alpha_{2,T}(u))^{\theta}(\alpha_{3,T}(u))^{1-\theta}(\alpha_{4,T}(u))^k+\alpha_{4,T}(z).
\end{equation}

We are now ready to apply the contraction argument to the operator (in the case $k=4$) \begin{equation}\label{contractionmapp}
    \Gamma_{u_0,z}(u(t))=V(t)u_{0}+\int_{0}^{t}V(t-t')\partial_{x}(u^{k+1})(t')dt'+z(t)
\end{equation}
on the ball $$B^{0}_{R,T}:=\{u:[0,T]\times\R\to\R |\,\norm{u}_{Y([0,T])}\leq R\}$$ where $z(t)$ is as in \eqref{stconv} and for $0<T<\infty$, we define $Y([0,T])$ as the solution space endowed with the norm
\begin{equation}\label{Y_space}
\norm{f}_{Y([0,T])}:=\max\{\Vert f \Vert_{C^{0}_{t}L^{2}_{x}([0,T])},\Vert \partial_{x}f \Vert_{C^{0}_{x} L^2_t([0,T])},\Vert f \Vert_{L^5_x L^{10}_t([0,T])}\},
\end{equation}
and $R=R(\omega)=2(c\norm{u_{0}}_{L^{2}_{x}}+\norm{z}_{Y([0,T])})$. Using estimates \eqref{alpha1u}--\eqref{alpha3u} and \eqref{alpha4u}, we deduce
\begin{align*}
    \norm{\Gamma_{u_0,z}(u)}_{Y([0,T])}\leq& c\norm{u_{0}}_{L^{2}_{x}}+c\norm{u}_{Y([0,T])}^{5}+\norm{z}_{Y([0,T])},
\end{align*}
and similarly we have
\begin{align*}
    \norm{\Gamma_{u_0,z}(u)-\Gamma_{u_0,z}(v)}_{Y([0,T])}\leq& c(\norm{u}_{Y([0,T])}^{4}+\norm{v}_{Y([0,T])}^{4})\norm{u-v}_{Y([0,T])}.
\end{align*}
Therefore, for $\omega$-a.s., choosing $R(\omega)$ as above guarantees the $\omega$ almost sure existence and uniqueness of the fixed point of the contraction map $\Gamma_{u_0,z}$ in $B_{R,T}^{0}$ provided that
\begin{equation}\label{smallness_condition}
    (2(c\norm{u_0}_{L^{2}_x}+\norm{z}_{Y([0,T])}))^{4}\leq \frac{1}{4c},\quad\text{$\omega$-a.s.,}
\end{equation}
which is possible by choosing $T>0$ sufficiently small. 

Next, we remove the smallness condition \eqref{smallness_condition}. Let $u_0\in L^{2}_{x}(\R)$. Then, in view of the proof of \cite[Theorem 2.10]{Kenig_gKdV_1993}, we deduce that for any $\varepsilon>0$, there exist $T_1=T_1(u_0,\varepsilon)$ and $\delta_1 =\delta_1 (u_0,\varepsilon)$ such that if $\norm{u_0-\widetilde{u}_0}_{L^{2}_{x}}<\delta_1$, then
\begin{equation}\label{smallness_of_tilde_u_0}
    \begin{aligned}
        \norm{\partial_{x}V(t)\widetilde{u}_0}_{L^{\infty}_{x}L^{2}_{t}([0,T_1])}<\varepsilon,\quad
        \norm{V(t)\widetilde{u}_{0}}_{L^{5}_{x}L^{10}_{t}}<\varepsilon.
    \end{aligned}
\end{equation}
Moreover, for any $\delta_2>0$ and any Schwartz function $\phi(x)$, we can find a Schwartz function $\widetilde{\phi}(x)$ such that
\begin{equation*}
    \sup_{\eta,\theta\in\N}\sup_{x\in\R}\vert x^{\eta}\langle\partial_{x}\rangle^{\theta}(\widetilde{\phi}-\phi)(x)\vert<\delta_{2}.
\end{equation*}
Therefore, by \cite[Proposition 2.2.6]{Grafakos_book}, we have 
\begin{equation}\label{bound_of_approx_phi}
    \Vert\widetilde{\phi}-\phi\Vert_{W^{s,p}_{x}}<c(p)\delta_2,\quad \text{for}\,\,s\in\R,\, 1\leq p\leq\infty.
\end{equation}
We shall denote
\begin{equation*}
    \widetilde{z}(t)=\int_{0}^{t}V(t-t')\widetilde{\phi}(x)g(t')dB(t').
\end{equation*} Then, from \eqref{bound_for_z(t)_with_expectation} and \eqref{bound_of_approx_phi} (letting $(s,p)=(m,2)$), for $T>0$, we have
\begin{equation}\label{smallness_of_tilde_z}
\begin{aligned}
    \max_{j\in \{2,3\}}\beta_{j,T}(\widetilde{z})&\lesssim \max_{j\in \{2,3\}}\beta_{j,T}(z-\widetilde{z})+\max_{j\in \{2,3\}}\beta_{j,T}(z)\\
    &\leq C(g)(T+T^{2+})\left(\Vert\phi-\widetilde{\phi}\Vert_{H^{m}_{x}}^{2}+\norm{\phi}_{H^{m}}^{2}\right)\\
    &\leq C(g)(T+T^{2+})(\delta_2^2+\norm{\phi}_{H^{m}}^{2})
\end{aligned}
\end{equation}
where $m$ is a sufficiently large positive integer (depending on the calculations in Section \ref{A priori estimates}). Next, we will show, for sufficiently small $T>0$, that $\max_{j\in\{ 2,3\}}\alpha_{j,T}(\widetilde{z})$ is arbitrarily small almost surely. First note that using Chebyshev's inequality with \eqref{smallness_of_tilde_z}, for any $\varepsilon>0$, we obtain the following bound:
\begin{equation}\label{chebine}
    \mathbb{P}\left(\max_{j\in \{2,3\}}\alpha_{j,T}(\widetilde{z})\geq \varepsilon\right)\leq \varepsilon^{-2}\max_{j\in \{2,3\}}\beta_{j,T}(\widetilde{z})\leq \varepsilon^{-2} C(g)(T+T^{2+})(\delta_2^2+\norm{\phi}_{H^{m}}^{2}).
\end{equation}
Now, given $\varepsilon>0$, we set \begin{align*}
 T_n=C(\varepsilon, \delta_2,g,\phi,m)2^{-n},\quad \Omega_{n}:=\left\{\max_{j\in \{2,3\}}\alpha_{j,T_n}(\widetilde{z})\geq \varepsilon\right\},\quad n\in\mathbb{N}.
\end{align*}Then, due to \eqref{chebine}, we infer that
\begin{equation*}
    \sum_{n\in\mathbb{N}}\mathbb{P}(\Omega_{n})<\infty.
\end{equation*}
Thus, as an application of the Borel-Cantelli lemma, we conclude that given $\varepsilon>0$, for $\omega$-a.s., there exists  $N(\omega)>0$ such that for all $n\geq N(\omega)$, we have
\begin{equation*}
    \max_{j\in \{2,3\}}\alpha_{j,T_{n}}(\widetilde{z})<\varepsilon\quad\text{$\omega$-a.s.}.
\end{equation*}
Note that due to \eqref{smallness_of_tilde_z}, the same argument holds for $\max_{j\in \{2,3\}}\alpha_{j,T_n}(z-\widetilde{z})$ as well. In particular, since $\max_{j\in \{2,3\}}\alpha_{j,T}(\cdot)$ is increasing in $T$ (see \eqref{spacetime_norms}), for $T_2\in(0,T_{N(\omega)}]$, we have
\begin{equation}\label{almost_sure_smallness_of_stoch_conv}
    \max_{j\in \{2,3\}}\alpha_{j,T_{2}}(\widetilde{z})<\varepsilon,\quad\max_{j\in \{2,3\}}\alpha_{j,T_{2}}(z-\widetilde{z})<\varepsilon\quad\text{$\omega$-a.s..}
\end{equation}
To complete the discussion of the removal of the smallness condition \eqref{smallness_condition}, we define the following norm:
\begin{equation}\label{new_alpha_1_norm}
     \widetilde{\alpha}_{1,T}(u):=\norm{u(t)-V(t)u_{0}-z(t)}_{L^{\infty}_{t}L^{2}_{x}([0,T])},
 \end{equation} 
by which we consider the ball
$$\widetilde{B}_{R,T}^{0}:=\{u:C^{0}_{t}L^{2}_{x}([0,T]\times\R) |\,\norm{u}_{\widetilde{Y}([0,T])}\leq R\}$$ 
where $R$ is to be determined later and
\begin{equation*}
   \norm{u}_{\widetilde{Y}([0,T])}=\max\{\widetilde{\alpha}_{1,T}(u),\alpha_{2,T}(u),\alpha_{3,T}(u)\}.
\end{equation*} 
Taking into account the discussion above, let us pick $\delta=\min\{\delta_{1},\delta_{2},\varepsilon\}$ and $T=T(\omega)=\min\{T_{1},T_{2}\}$, and assume that $\max\{\norm{u_{0}-\widetilde{u}_0}_{L^{2}_{x}},\Vert\phi-\widetilde{\phi}\Vert_{H^{m}_x}\}<\delta$. Let $u\in\widetilde{B}_{R,T}^{0}$. Using the estimate \eqref{est:Kato_2} with $(p,q)=(\infty,2)$ and H\"older's inequality, we have:
\begin{equation*}
   \norm{\int_{0}^{t}V(t-t')(u^{5})_{x}(t')dt'}_{L^{2}_{x}}\leq c\norm{u^5}_{L^{1}_{x}L^{2}_{t}([0,T])}\leq c(\alpha_{3,T}(u))^{5}
\end{equation*}
which implies that
\begin{equation}\label{estimation_tilde_alpha_1}
    \widetilde{\alpha}_{1,T}(\Gamma_{u_0,z}(u))=\norm{\int_{0}^{t}V(t-t')(u^{5})_{x}(t')dt'}_{L^{\infty}_{t}L^{2}_{x}([0,T])}\leq c\norm{u}_{\widetilde{Y}([0,T])}^{5}.
\end{equation}
 We use the triangle inequality to obtain:
\begin{multline*}\label{estimation_for_duhamel_term}
  \alpha_{3,T}(\Gamma_{u_0,z}(u))\leq\norm{V(t)(u_0-\widetilde{u}_0)}_{L^{5}_{x}L^{10}_{t}([0,T])}+\norm{V(t)\widetilde{u}_0}_{L^{5}_{x}L^{10}_{t}([0,T])}\\+\norm{z-\widetilde{z}}_{L^{5}_{x}L^{10}_{t}([0,T])}
    +\norm{\widetilde{z}}_{L^{5}_{x}L^{10}_{t}([0,T])}+\norm{\int_{0}^{t}V(t-t')(u^{5})_{x}(t')dt'}_{L^{5}_{x}L^{10}_{t}([0,T])}.
\end{multline*}
Implementing the estimates \eqref{est:Kato_1}, \eqref{smallness_of_tilde_u_0}, and \eqref{almost_sure_smallness_of_stoch_conv} in turn for the first four terms on the right side of above estimate, we get
\begin{equation*}\label{estimation_for_duhamel_term-2}
    \alpha_{3,T}(\Gamma_{u_0,z}(u))\leq c\norm{u_{0}-\widetilde{u}_0}_{L^{2}_{x}}+c\varepsilon+\norm{\int_{0}^{t}V(t-t')(u^{5})_{x}(t')dt'}_{L^{5}_{x}L^{10}_{t}([0,T])}\quad\text{$\omega$-a.s..}
\end{equation*}
Applying \eqref{Kato_estimate_Dodson} (considering the inhomogeneous problem \eqref{eq:inhom_airy} with $u_0=0$ and $g=(u^5)_x$) with $(p_2,q_2)=(\infty,2)$ to the term with integral part in the above estimate and the H\"older's inequality, we ultimately get the following:
\begin{equation}\label{estimations_alpha_2}
\begin{aligned}
    \alpha_{3,T}(\Gamma_{u_0,z}(u))\leq c\varepsilon+c\Vert D^{-1}_x(u^5)_x\Vert_{L^1_xL^2_t([0,T])} \leq c\varepsilon+c\norm{u}_{\widetilde{Y}([0,T])}^{5}\quad\text{$\omega$-a.s..}
    \end{aligned}
\end{equation}
Similarly, we also obtain:
\begin{equation}\label{estimations_alpha_3}
    \alpha_{2,T}(\Gamma_{u_0,z}(u))\leq c\varepsilon+c\norm{u}_{\widetilde{Y}([0,T])}^{5}\quad\text{$\omega$-a.s.}.
\end{equation}
Therefore, combining \eqref{estimation_tilde_alpha_1}, \eqref{estimations_alpha_2}, and \eqref{estimations_alpha_3}, almost surely, we have
\begin{equation*}\label{estimations_for_self map}
        \norm{\Gamma_{u_0,z}(u)}_{\widetilde{Y}([0,T])}\leq c\varepsilon+c\norm{u}_{\widetilde{Y}([0,T])}^{5}.
\end{equation*}
Therefore, whenever
\begin{equation}\label{self_map_condition}
    c\varepsilon+cR^{5}< R
\end{equation}
we have $\Gamma_{u_0,z}(\widetilde{B}_{T,R}^{0})\subseteq \widetilde{B}_{T,R}^{0}$. Moreover, for $u, \widetilde{u}\in\widetilde{B}_{R,T}^{0}$, similar computations as in \eqref{estimation_tilde_alpha_1}, \eqref{estimations_alpha_2}, and \eqref{estimations_alpha_3} yield
\begin{equation}\label{estimations_for_contraction map}
    \norm{\Gamma_{u_0,z}(u)-\Gamma_{u_0,z}(\widetilde{u})}_{\widetilde{Y}([0,T])}\leq c(\norm{u}^{4}_{\widetilde{Y}([0,T])}+\norm{\widetilde{u}}_{\widetilde{Y}([0,T])}^{4})\norm{u-\widetilde{u}}_{\widetilde{Y}([0,T])}.
\end{equation}
Then, as $u, \widetilde{u}\in\widetilde{B}_{R,T}^{0}$, taking
\begin{equation}\label{contraction_map_condition}
    2cR^{4}<\frac{1}{2}
\end{equation}
in \eqref{estimations_for_contraction map} implies that $\Gamma_{u_0,z}$ is a contraction map. Therefore, in order for \eqref{self_map_condition} and \eqref{contraction_map_condition} to hold simultaneously, by fixing $\varepsilon>0$ such that $(2c)^{5}\varepsilon^{4}<\frac{1}{2}$, we let $R=2c\varepsilon$. Lastly, the continuous dependence on the initial data follows similarly; we refer to \cite{deBouard_H1, Kenig_gKdV_1993}.
\begin{Remark}
    The above discussion has demonstrated that we choose the existence time $T>0$ irrespective of the initial data $u_0$ in the contraction mapping argument under the smallness condition \eqref{smallness_condition}. Moreover, the smallness condition \eqref{smallness_condition} does not only depend on the smallness of the data but also on the smallness of the norm $\norm{z}_{Y([0,T])}$, which is provided when $T$ is taken sufficiently small (similar to the argument to obtain \eqref{almost_sure_smallness_of_stoch_conv}). In removing \eqref{smallness_condition}, we have seen that the existence time needs to be dependent on the position of the initial data and the components of the stochastic convolution. Thus, in contrast to many subcritical equations, where the local existence time can often depend only on the size of the data, in our discussion, however, the argument for the removal of the smallness condition forces us to incorporate more dependencies into the existence time $T$.
\end{Remark}
In the case $k>4$, $\omega$ almost sure existence and uniqueness of the fixed point of $\Gamma_{u_0,z}$ defined in \eqref{contractionmapp} follows in a similar way, in which case our argument depends on the resolution spaces $Z$ defined by the following norm:
\begin{multline}\label{Z_space}
    \norm{u}_{Z}:=\max\{\norm{D_{x}^{s_k}u}_{C^{0}_{t}L^{2}_{x}([0,T])},\norm{D^{1+s_k}_{x}u}_{C^{0}_{x}L^{2}_{t}([0,T])},\norm{D_{x}^{s_k}u}_{L^{5}_{x}L^{10}_{t}([0,T])},\\
    \Vert D_x^{\frac{1}{10}-\frac{2}{5k}} D_t^{\frac{3}{10}-\frac{6}{5k}}  u\Vert_{L^{p_{k}}_{x} L^{q_{k}}_{t}([0,T])}\}.
\end{multline}

\subsection{Global existence of solutions}\label{glexst}
 We start by noting the basic fact that when $k$ is an even integer, the energy \eqref{energy_functional} turns into a sign-definite functional, which is of fundamental importance to begin with our present discussion (see the assumption of $(ii)$ in Theorem \ref{thm:well-posedness}). We also point out another basic fact that, unlike the deterministic analogue of \eqref{eq:stochastic_gKdV}, the mass \eqref{mass_functional} and the energy \eqref{energy_functional} associated with the Cauchy problem \eqref{eq:stochastic_gKdV} are not conserved due to the additional stochastic forcing term appearing in \eqref{eq:stochastic_gKdV}. In contrast with the failure of conservation at the $L^2$ and $H^1$ levels, we can still show that the $r$-th moments of the mass and energy remain bounded for all time and for all $r\geq 1$:
\begin{equation}\label{momenmass}
    \mathbb{E}\left[\sup\limits_{0\leq t<\infty}(M(u(t))+E(u(t)))^{r}\right]\leq C(M(u_{0}),E(u_{0}),r,\phi,g),
\end{equation}
where $\phi$ and $g$ are introduced below \eqref{eq:stochastic_gKdV}. In order to show \eqref{momenmass}, instead of applying Itô's lemma to the typical elements $(M(u(t)+E(u(t)))^{m}$ (where $m\in\mathbb{Z}_+$ is sufficiently large), we apply it to $(C(\phi)+M(u(t))+E(u(t)))^{m}$ to derive a much simplified bound (as in \eqref{boundforlem23}), where $C(\phi)$ is a constant consisting of various norms of $\phi$. Therefore,
\begin{equation}\label{ito_for_moment_of_mass_energy}
    \begin{aligned}
    (C(\phi)+M(u(t))+E(u(t)))^{m}=&(C(\phi)+M(u(0))+E(u(0)))^{m}\\
    &+m\int_{0}^{t}(C(\phi)+M(u(t'))+E(u(t')))^{m-1}F_{2}(t')dt'\\
    &+\frac{m(m-1)}{2}\int_{0}^{t}(C(\phi)+M(u(t'))+E(u(t')))^{m-2}F_{1}^{2}(t')dt'\\
    &+m\int_{0}^{t}(C(\phi)+M(u(t'))+E(u(t')))^{m-1}F_{1}(t')dB(t'),
    \end{aligned}
\end{equation}
where
\begin{equation*}
    F_{1}(t')=\int_{\mathbb{R}}(2u(t')\phi+u^{k+1}(t')\phi+u_{x}(t')\phi_{x})g(t')dx,
\end{equation*}
and
\begin{equation*}
    F_{2}(t')=\int_{\mathbb{R}}\left(\frac{k+1}{2}u^{k}(t')\phi^{2}+\phi^{2}+\phi_{x}^{2}\right)g^{2}(t')dx.
\end{equation*}
Note that the stochastic integral \begin{align*}
    m\int_{0}^{t}(C(\phi)+M(u(t'))+E(u(t')))^{m-1}F_{1}(t')dB(t')=:\mathcal{M}_{t}
\end{align*} 
appearing in \eqref{ito_for_moment_of_mass_energy} is a continuous martingale in $t$, which is necessary to apply Lemma \ref{gronwall}. We first intend to estimate $F_{1}$ and $F_{2}$. Thus applying Hölder inequality in the $x$ variable and then Young inequality, we have
\begin{align*}
    F_{1}(t')\leq&(2\Vert\phi\Vert_{L^{2}_{x}}\Vert u(t')\Vert_{L^{2}_{x}}+\Vert\phi\Vert_{L^{k+2}_{x}}\Vert u(t')\Vert_{L^{k+2}_{x}}^{k+1}+\Vert\phi_{x}\Vert_{L^{2}_{x}}\Vert u_{x}(t')\Vert_{L^{2}_{x}})\vert g(t')\vert \\
    \leq& (C(\phi)+M(u(t'))+E(u(t')))\vert g(t')\vert.
\end{align*}
Similarly, $F_{2}(t')$ is shown to be bounded by the following:
\begin{align*}
    F_{2}(t')\leq (C(\phi)+M(u(t'))+E(u(t')))g^{2}(t').
\end{align*}
Using these bounds on $F_1$ and $F_2$, we get
\begin{equation}
\begin{aligned}\label{boundforlem23}
    (C(\phi)+M(u(t))+E(u(t)))^{m}\leq&(C(\phi)+M(u(0))+E(u(0)))^{m}\\
    &+\frac{m(m+1)}{2}\int_{0}^{t}(C(\phi)+M(u(t'))+E(u(t')))^{m}g^{2}(t')dt'+\mathcal{M}_{t}.
\end{aligned}
\end{equation}
Now applying Lemma \ref{gronwall} for \eqref{boundforlem23} with any $0<q<p<1$ and stopping time $\tau>0$, we obtain
\begin{align*}
    \left(\mathbb{E}\left[\sup\limits_{0\leq t\leq\tau}(C(\phi)+M(u(t))+E(u(t)))^{mq}\right]\right)^{1/q}&\leq\left(\frac{p}{p-q}\right)^{1/q}C(M(u_{0}),E(u_{0}),m,\phi)\\&\times\left(\mathbb{E}\left[\exp\left(\frac{pm(m+1)}{2(1-p)}\int_{0}^{\tau}g^{2}(t')dt'\right)\right]\right)^{\frac{1-p}{p}}\\&=:C(M(u_{0}),E(u_{0}), m, q, \phi, g).
\end{align*}
Note that since $g\in L^{2}_{t}([0,\infty))$, $\omega$-a.s., the above bound is uniform in $\tau$. Also, since $C(\phi)$ is a positive deterministic constant, the above bound is valid for moments of mass and energy. Thus, setting $mq=r$ one can deduce \eqref{momenmass}, because choosing $m$ sufficiently large, we may have $mq\geq 1$. Combining this with the local well-posedness results of Section \ref{locwelres}, we conclude (as in the proof of \cite[Theorem 1.2]{Herr_2019}) $\omega$ almost surely that
\begin{equation*}
    \sup_{t\in[0,\infty)}\norm{u(t)}_{L^{2}_{x}}<\infty,
\end{equation*}
when $k=4$ and
\begin{equation*}
    \sup_{t\in[0,\infty)}\norm{u(t)}_{H^{1}_{x}}<\infty,
\end{equation*}
provided even integer $k>4$. This completes the proof of Theorem \ref{thm:well-posedness}.
\section{Estimates on the tail of the stochastic convolution}\label{section:est_on_tail}\label{chap4}
In this section, we are concerned with $\omega$ almost sure global-in-time boundedness of the space-time norms $L^{p}_{x}L^{q}_{t}([T,\infty))$ and $L^{q}_{t}L^{p}_{x}([T,\infty))$, satisfying the Strichartz and Kato admissibility conditions (see \eqref{eq:admissible_pair} and \eqref{eq:Kato_admissible}), of the tail of the stochastic convolution $z_{*}(t)$ (defined in \eqref{tail_of_stochastic_convolution}), $t\in[T,\infty)$ where $T\gg 1$ is a large time parameter. Based on observations in our previous work for the stochastic nonlinear Schr\"odinger equations \cite{Basakoglu_Scattering_NLS} (the situation with regard to Strichartz estimates is the same in the current paper since $\phi$ is in the Schwartz class, and the operators $\langle D_{x}\rangle^{s}$ and $V(t)$, defined in \eqref{airy_propagator}, commutes. Therefore, we refer to Section $2.2$ in \cite{Basakoglu_Scattering_NLS} for details), we have
\begin{equation}\label{tailest}
    \Vert z_{*}(t)\Vert_{W^{s,p}_{x}}\lesssim\langle t\rangle^{-\gamma+\frac{1}{2}+}\quad \text{-a.s.}
\end{equation}
where $s\geq 0$ and $\vert g(t)\vert=o(t^{-\gamma})$ $\omega$ almost surely as $t\to\infty$. Let $(p,q)$ be an admissible pair satisfying \eqref{eq:admissible_pair}. Then using \eqref{tailest}, we have:
\begin{align*}
    \Vert z_{*}\Vert_{L^{q}_{t}W^{s,p}_{x}([T,\infty))}
    \lesssim\Big(\int_{T}^{\infty}\vert t\vert^{q(-\gamma+\frac{1}{2}+)}dt\Big)\lesssim 1
\end{align*}
uniformly in $T\gg 1$ provided that $\gamma>\frac{1}{2}+\frac{1}{q}$. Note that in order not to impose an extra time decay on $g(t)$, we may take $\beta=0$, depending on the fact that the pair $(p,q)$ satisfies the relation \eqref{eq:admissible_pair} for $2\leq p\leq \infty$ and $\beta\in[0,\frac{1}{2}]$. Then we necessarily require $\gamma>\frac{2}{3}$. The main result of this section is as follows:

\begin{proposition}\label{mainprop}
Let $(p,q,\alpha)$ be a Kato triple. Let $g(t)=(1+|t|)^{-\gamma}$ with $\gamma>\frac{2}{3}$.  For $T\gg 1$, we have  \begin{equation*}\label{kato_type_est_tail}
	\begin{aligned}
		\|D^{\alpha}_xz_*(t,x)\|_{L^p_xL^q_t([T,\infty))}\leq C_{p, q, T}C(\phi)\quad \text{for $\omega$-a.s.},
	\end{aligned}	
\end{equation*}
where $C(\phi)$ consists of various norms of $\phi$. 
\end{proposition} We first need the corresponding oscillatory integral estimates.
\subsection{Oscillatory integral estimates} In this subsection, for $b\geq 2, b\in \N$ and suitable values of $\alpha\in\R$, we will estimate the oscillatory integral $I^{b,\alpha}(x)$ defined in \eqref{osc1} in terms of $x\in \mathbb{R}$. We start by observing that the following argument gives a lower bound on $\alpha$ for $I^{b,\alpha}$ to make sense. To this end, let $x$ be fixed and $K_x$ be defined in the following
\begin{equation*}
	0\leq \xi\leq  K_x, \qquad K_x=\frac{1}{10(1+|x|)}
\end{equation*}
so that we have the simple bound:
\begin{equation*}
	|\xi^b+x\xi|\leq \frac{1}{10^b}+\frac{1}{10}\leq \frac{1}{5}.
\end{equation*}
Then we have
\begin{equation*}
	\begin{aligned}
		\Big|\int_{0}^{K_x} e^{i(\xi^b+x\xi)}|\xi|^{\alpha}d\xi\Big|\geq \Big|\int_{0}^{K_x} \cos (\xi^b+x\xi)|\xi|^{\alpha}d\xi\Big|\geq \frac{1}{2}\int_{0}^{K_x} \xi^{\alpha}d\xi.
	\end{aligned}	
\end{equation*}
This imposes the condition $\alpha>-1$. Another condition emerges from decay at infinity. For simplicity, pick $x=0$. Then, by a change of variables $\eta=\xi^b$, we write
\begin{equation*}
	\begin{aligned}
		\Big|\int_{0}^{\infty} e^{i\xi^b}|\xi|^{\alpha}d\xi\Big|\geq  \frac{1}{b}\Big|\sum_{k=0}^{\infty}\int_{k\pi}^{(k+1)\pi} |\eta|^{\frac{\alpha+1}{b}-1}\sin \eta d\eta\Big|.
	\end{aligned}	
\end{equation*}
Here we must have $\frac{\alpha+1}{b}-1< 0$, otherwise the infinite series above would not converge. This imposes the condition $\alpha<b-1$. 
Once these two conditions are satisfied, the integral makes sense and satisfies the following decay estimates. 

\begin{proposition}\label{propIalphab}
	Let $\alpha\in  \R, b\in  \N$, and $b\geq 2$. Then
	\begin{equation*}
		\begin{aligned}
			|I^{b,\alpha}(x)|=\Big|\int_{\R} e^{i[\xi^b+x\xi]}|\xi|^{\alpha}d\xi\Big| \lesssim_{\alpha,b}\begin{cases} (1+|x|)^{-1-\alpha}  & \quad  -1<\alpha\leq  -1/2,  \\  (1+|x|)^{\frac{1}{b-1}(\alpha-\frac{b}{2}+1)}  & \quad  -1/2<\alpha<  b-1. \end{cases}
		\end{aligned}	
	\end{equation*}
\end{proposition}
We scale this estimate to obtain the following corollary: 
\begin{Corollary}\label{remark:osc_int_time_decay} With the same $\alpha$ and $b$ as in Proposition \ref{propIalphab}, we have
\begin{equation*}
\begin{aligned}
	\Big|\int_{\R} e^{i[t\xi^b+x\xi]}|\xi|^{\alpha}d\xi\Big| &\lesssim_{\alpha} \begin{cases}
		(|t|^{1/b}+|x|)^{-\alpha-1} \quad  &   \quad  -1<\alpha\leq -1/2,  \\ 	(|t|^{1/b}+|x|)^{\frac{1}{b-1}(\alpha-\frac{b}{2}+1   )}|t|^{\frac{1}{b-1}(-\alpha-\frac{1}{2})} \quad  &   \quad  -1/2<\alpha<b-1.
	\end{cases}
\end{aligned}
\end{equation*}
\end{Corollary}
\begin{proof}
  It suffices to write 
  $$\Big|\int_{\R} e^{i[t\xi^b+x\xi]}|\xi|^{\alpha}d\xi\Big|=|t|^{-\frac{\alpha+1}{b}}\Big|\int_{\R} e^{i[\xi^b+x\xi]}|\xi|^{\alpha}d\xi\Big|$$
  and apply the Proposition \ref{propIalphab}.
\end{proof}
Hence, as a result of Corollary \ref{remark:osc_int_time_decay}, we also obtain an extra decay in the temporal variable.
\begin{proof}[Proof of Proposition \ref{propIalphab}]
We start with showing that for $|x|\leq 100$,  we can bound $I^{b,\alpha}$ by a constant. We write 
\begin{equation*}
	\begin{aligned}
		 I^{b,\alpha}(x)=	\int_{|\xi|<100} +\int_{\xi\geq 100}+\int_{\xi\leq -100} e^{i[\xi^b+x\xi]}|\xi|^{\alpha}d\xi =I_1+I_2+I_3,
	\end{aligned}
\end{equation*}
and first estimate $I_{1}$ by the following:
\begin{equation}\label{ibpn0}
	\begin{aligned}
		|I_1|\leq \int_{|\xi|<100} |\xi|^{\alpha}d\xi =2\frac{ 100^{1+\alpha}}{1+\alpha}.
	\end{aligned}
\end{equation}
Then, integration by parts applied to
\begin{equation*}
	\begin{aligned}
		I_2=	\int_{\xi\geq 100}e^{i[\xi^b+x\xi]}i(b\xi^{b-1}+x)\frac{\xi^{\alpha}}{i(b\xi^{b-1}+x)}d\xi
	\end{aligned}
\end{equation*}
yields
\begin{equation}\label{ibpp2}
	\begin{aligned}
		|I_2|  \leq  \Big|\frac{10^{2\alpha}}{2\cdot10^2+x}\Big|+\int_{100}^{\infty}\Big|\frac{\alpha\xi^{\alpha-1}}{b\xi^{b-1}+x}\Big|+\Big|\frac{b(b-1)\xi^{\alpha+b-2}}{(b\xi^{b-1}+x)^2} \Big|d\xi.
	\end{aligned}
\end{equation}
Since $\xi\geq 100$ and $x\geq -100$, we have $b\xi^{b-1}+x\geq \xi^{b-1} \geq 100$, which allows us to conclude
\begin{equation}\label{ibpn3}
	\begin{aligned}
		\leq  10^{2\alpha-2}+\int_{100}^{\infty}|\alpha|\xi^{\alpha-b}+\Big|\frac{b(b-1)\xi^{\alpha+b-2}}{\xi^{2b-2}} \Big|d\xi  \leq  {10^{2\alpha-2}}-(|\alpha|+b^2)\frac{100^{\alpha-b+1}}{\alpha-b+1}.
	\end{aligned}
\end{equation}
By a change of variables, we write
\begin{equation*}
	\begin{aligned}
		I_3=	\int_{\xi\geq 100}e^{i[(-1)^b\xi^{b}-x\xi]}i[b(-1)^b\xi^{b-1}-x]\frac{\xi^{\alpha}}{i[b(-1)^b\xi^{b-1}-x]}  d\xi
	\end{aligned}
\end{equation*}
for which we apply integration by parts to have:
\begin{equation}\label{ibpn1}
	\begin{aligned}
		|I_3| &\leq  10^{2\alpha-2}+\int_{100}^{\infty}(|\alpha|+b^2)\xi^{\alpha-b} d\xi  \leq  10^{2\alpha-2}-(|\alpha|+b^2)\frac{100^{\alpha-b+1}}{\alpha-b+1}.
	\end{aligned}
\end{equation}	
Therefore, by \eqref{ibpn0}--\eqref{ibpn1}, we have 
\begin{align}\label{|x|<100}
    |I^{b,\alpha}(x)|\leq C(\alpha, b),\quad \text{for}\,\,|x|\leq 100.
\end{align}
Next, we consider the case $|x|\geq 100$. There are three main types of contributions that come into this oscillatory integral. The first comes from the origin, which is a singularity of $|\xi|^{\alpha},\alpha<0.$ The second is the stationary points, that is, the set of points
\begin{align}\label{stationarypts}
  S_x^b=\{\xi\in\R:\partial_{\xi}[\xi^b+x\xi]=0\},\quad x\in\mathbb{R}.  
\end{align}
 If $b$ is odd, then  $S_x^b=\{\pm(-x/b)^{\frac{1}{b-1}}\}$ for $\,x\in\mathbb{R}_-$; while if $b$ is even, then for any $x\in\R$ the set   $S_x$  contains only $\xi=(-x/b)^{\frac{1}{b-1}}$. The final contribution comes from  $\xi$  large in absolute value. With a view towards this, we divide our integral into three pieces
\begin{equation}\label{J123}
	\begin{aligned}
		I^{b,\alpha}(x)=	\int_{|\xi|<\frac{|x|^{\frac{1}{b-1}}}{2b}} +\int_{\frac{|x|^{\frac{1}{b-1}}}{2b}\leq |\xi| \leq |x|^{\frac{1}{b-1}}}+\int_{|\xi|>{|x|^{\frac{1}{b-1}}}} e^{i[\xi^3+x\xi]}|\xi|^{\alpha}d\xi =I_1^*+I_2^*+I_3^*.
	\end{aligned}
\end{equation}	
We start with $I_1^*$. As the associated interval of integration contains a singularity, the usual theorems on stationary/nonstationary phase, like the van der Corput lemma (Lemma \ref{lemma:van der Corput}), do not apply. As such, we subdivide this interval into three: 	
\begin{equation}\label{J1}
	\begin{aligned}
		I_1^*=	\int_{|\xi|\leq \frac{1}{b|x|}}+	\int_{\frac{1}{b|x|}<\xi<\frac{|x|^{\frac{1}{b-1}}}{2b}}+	\int_{-\frac{|x|^{\frac{1}{b-1}}}{2b}<\xi< -\frac{1}{b|x|} } e^{i[\xi^b+x\xi]}|\xi|^{\alpha}d\xi =I_{11}^*+I_{12}^*+I_{13}^*.
	\end{aligned}
\end{equation}
Consider $I_{11}^*$ first. In this interval of integration, as $\xi^b \ll x\xi<1$, there is no significant oscillation that leads to cancellation,  so we may directly bring the absolute value inside without much loss:
\begin{equation*}
	\begin{aligned}
		|I_{11}^*|	 \leq \int_{|\xi|\leq 1/b|x|} |\xi|^{\alpha}d\xi =2  \int_{0}^{ 1/b|x|} \xi^{\alpha}d\xi=\frac{2}{(\alpha+1)b^{\alpha+1}}|x|^{-\alpha-1}.
	\end{aligned}
\end{equation*}
Now for $I_{12}^*$, there is neither a singularity nor a stationary point. So we expect a good cancellation. Integration by parts yields
\begin{equation*}\label{}
	\begin{aligned}
		I_{12}^*&= \frac{\xi^{\alpha}e^{i[\xi^b+x\xi]}}{i(b\xi^{b-1}+x)}\Big|^{\frac{|x|^{\frac{1}{b-1}}}{2b}}_{1/b|x|}-\frac{1}{i}\int_{1/b|x|}^{\frac{|x|^{\frac{1}{b-1}}}{2b}}e^{i[\xi^b+x\xi]}\Big[\frac{\alpha\xi^{\alpha-1}}{b\xi^{b-1}+x}-\frac{b(b-1)\xi^{\alpha+b-2}}{(b\xi^{b-1}+x)^2} \Big]d\xi.
	\end{aligned}
\end{equation*}
We observe that in the interval $[1/b|x|,{|x|^{\frac{1}{b-1}}}/{2b}]$, we have $|b\xi^{b-1}+x|\geq |x|/2$. Therefore
\begin{equation*}
	\begin{aligned}
		| I_{12}^*| \leq&   \frac{2|x|^{\frac{\alpha}{b-1}}}{(2b)^{\alpha}|x|}+\frac{2|x|^{-\alpha}}{b^{\alpha}|x|}+\int_{1/b|x|}^{\frac{|x|^{\frac{1}{b-1}}}{2b}}\Big[\frac{2\alpha\xi^{\alpha-1}}{|x|}+\frac{4b(b-1)\xi^{\alpha+b-2}}{|x|^2} \Big]d\xi  \\ 
        \leq& 2^{1-\alpha}b^{-\alpha}|x|^{\frac{\alpha}{b-1}-1} +2b^{-\alpha}|x|^{-\alpha-1}+\frac{2}{|x|}\Big( \frac{|x|^{\frac{\alpha}{b-1}}}{(2b)^{\alpha}}+b^{-\alpha}|x|^{-\alpha}   \Big) \\
        &+\frac{4b(b-1)}{| \alpha +b-1||x|^2}\Big( \frac{|x|^{\frac{\alpha+b-1}{b-1}}}{(2b)^{\alpha+b-1}}+(b|x|)^{-\alpha-b+1}  \Big)\\ 
               \lesssim_{\alpha,b}&  |x|^{\frac{\alpha}{b-1}-1} +|x|^{-\alpha-1}.
	\end{aligned}
\end{equation*}
The estimate for $I_{13}^*$ is similar to that of $I_{12}^*$. Indeed, changing variables and implementing an integration by parts to $I_{13}^*$, one can get	
\begin{equation*}\label{}
	\begin{aligned}
		|I_{13}^*| \lesssim_{\alpha,b}     |x|^{\frac{\alpha}{b-1}-1} +|x|^{-\alpha-1}.  
	\end{aligned}
\end{equation*}
Next, we turn our attention to the term $I_2^*$ and decompose it into two:
\begin{equation*}
	\begin{aligned}
		I_2^*=\int_{\frac{|x|^{\frac{1}{b-1}}}{2b}} ^{ |x|^{\frac{1}{b-1}}}+ \int^{-\frac{|x|^{\frac{1}{b-1}}}{2b}} _{ -|x|^{\frac{1}{b-1}}} e^{i[\xi^b+x\xi]}|\xi|^{\alpha}d\xi=I_{21}^*+I_{22}^*.
	\end{aligned}
\end{equation*}
Regarding $I_{21}^*$, we take account of the case $x>0$ first. In this case, as we have discussed: $$S_x^b\cap [x^{\frac{1}{b-1}}/2b, x^{\frac{1}{b-1}}]= \emptyset,$$ where $S_x^b$ is as in \eqref{stationarypts}, thus applying integration by parts as before yields
\begin{equation*}\label{}
	\begin{aligned}
		|I_{21}^*| \lesssim_{\alpha,b}     |x|^{\frac{\alpha}{b-1}-1}.  
	\end{aligned}
\end{equation*}
 On the other hand, in the case $x< 0$, $b\xi^{b-1}+x=0$ can occur within the range under consideration. We shall define $\phi_x(\xi)=\xi^b|x|^{-\frac{b-2}{b-1}}-|x|^{\frac{1}{b-1}}\xi$ so that 
\begin{equation*}
	\begin{aligned}
		I_{21}^*=\int_{\frac{|x|^{\frac{1}{b-1}}}{2b}} ^{ |x|^{\frac{1}{b-1}}}e^{i|x|^{\frac{b-2}{b-1}}\phi_x(\xi)}\xi^{\alpha}d\xi.
	\end{aligned}
\end{equation*}
Note that $\phi'_x(\xi)=b\xi^{b-1}|x|^{-\frac{b-2}{b-1}}-|x|^{\frac{1}{b-1}}$ and $\phi''_x(\xi)=b(b-1)\xi^{b-2}|x|^{-\frac{b-2}{b-1}}$. Then, we have $\phi''_x(\xi)\geq \frac{b(b-1)}{(2b)^{b-2}}$ for $\xi\in[|x|^{\frac{1}{b-1}}/2b, |x|^{\frac{1}{b-1}}]$ 
We shall apply Lemma \ref{lemma:van der Corput} by letting $f(\xi)=\xi^{\alpha}$. First, note that
\begin{align*}
  \|f\|_{L^{\infty}[{\frac{|x|^{\frac{1}{b-1}}}{2b}} ,{ |x|^{\frac{1}{b-1}}}]}=\begin{cases}
   |x|^{\alpha/(b-1)}/(2b)^{\alpha}\quad&\text{if}\,\,\alpha \leq 0, \\   |x|^{\alpha/(b-1)}  \quad&\text{if}\,\,\alpha>0, 
  \end{cases}  
\end{align*}
and 
\begin{equation*}
	\begin{aligned}
		\|f'\|_{L^{1}[{\frac{|x|^{\frac{1}{b-1}}}{2b}} ,{ |x|^{\frac{1}{b-1}}}]} =|1-(2b)^{-\alpha}||x|^{\frac{\alpha}{b-1}}.
	\end{aligned}
\end{equation*}
Thus, by van der Corput lemma
\begin{equation*}
	\begin{aligned}
		|	I_{21}^*|\leq \Big| \int_{\frac{|x|^{\frac{1}{b-1}}}{2b}}^{ |x|^{\frac{1}{b-1}}}e^{{i|x|^{\frac{b-2}{b-1}}}\phi_x(\xi)}\xi^{\alpha}d\xi  \Big| \leq C_b |x|^{-\frac{b-2}{2b-2}}(\|f\|_{L^{\infty}}+\|f'\|_{L^{1}})=2C_b(1+(2b)^{-\alpha})|x|^{-\frac{b-2}{2b-2}+\frac{\alpha}{b-1}},
    \end{aligned}
\end{equation*}
where $C_b$ is independent of $\phi_x$, hence $x$, so depends only on $b$. Next, we turn to the estimation of $I_{22}^*$ and write
\begin{equation*}
	\begin{aligned}
		I_{22}^*=\int_{\frac{|x|^{\frac{1}{b-1}}}{2b}}^{ |x|^{\frac{1}{b-1}}} e^{i[(-1)^b\xi^b-x\xi]}\xi^{\alpha}d\xi.
	\end{aligned}
\end{equation*}	
According to the parity of $b$ and the sign of $x$, there are four situations, but they can be handled in two cases. First, we assume $b$ odd and $x>0$ or $b$ even and $x<0$. In these two situations, stationary points do not exist, hence we apply integration by parts to arrive at
\begin{equation*}\label{}
	\begin{aligned}
		|I_{22}^*|	\lesssim_{\alpha,b}  |x|^{\frac{\alpha}{b-1}-1}.
	\end{aligned}
\end{equation*}	
Secondly, we assume $b$ odd and $x<0$, or $b$ even and $x>0$. Then $b(-1)^b\xi^{b-1}-x=0$ is possible. Let $\psi_x(\xi)=(-\xi)^{b}|x|^{-\frac{b-2}{b-1}}+(-1)^{b-1}|x|^{\frac{1}{b-1}}\xi$. Then $\psi'_x(\xi)=-b(-\xi)^{b-1}|x|^{-\frac{b-2}{b-1}}+(-1)^{b-1}|x|^{\frac{1}{b-1}}$ and $\psi''_x(\xi)=b(b-1)(-\xi)^{b-2}|x|^{-\frac{b-2}{b-1}}$. We have $|\psi''_x(\xi)|\geq \frac{b(b-1)}{(2b)^{b-2}}$ on the integration interval of $I_{22}^*$. Thus by Lemma \ref{lemma:van der Corput} 
\begin{equation*}
	\begin{aligned}
		|	I_{22}^*|\leq \Big|\int_{\frac{|x|^{\frac{1}{b-1}}}{2b}} ^{|x|^{\frac{1}{b-1}}} e^{i|x|^{\frac{b-2}{b-1}}\psi_x(\xi)}\xi^{\alpha}d\xi \Big| \leq C_b |x|^{-\frac{b-2}{2b-2}}(\|f\|_{L^{\infty}}+\|f\|_{L^{1}})=2C_b(1+(2b)^{-\alpha})|x|^{-\frac{b-2}{2b-2}+\frac{\alpha}{b-1}}.
	\end{aligned}
\end{equation*}
It remains to consider $I_3^*$. Divide this integral into two parts by
\begin{equation}\label{J3}
	\begin{aligned}
		I_3^*=\int_{\xi>|x|^{\frac{1}{b-1}}}+\int_{\xi<-|x|^{\frac{1}{b-1}}} e^{i[\xi^b+x\xi]}|\xi|^{\alpha}d\xi =I_{31}^*+I_{32}^*.
	\end{aligned}
\end{equation}
Note that over these intervals, there is no stationary point. Thus, as before, employing integration by parts together with lower bounds on the size of the  phase yields
\begin{equation*}\label{}
	\begin{aligned}
		|I_{31}^*|,|I_{32}^*|\lesssim_{\alpha,b} |x|^{\frac{\alpha}{b-1}-1}.
	\end{aligned}
\end{equation*}

Combining all of the above estimates for \eqref{J123} with \eqref{J1}--\eqref{J3}, we obtain 
\begin{equation}\label{|x|>100}
	|I^{b,\alpha}(x)|\leq C_{\alpha,b}\Big[ |x|^{\frac{\alpha}{b-1}-\frac{b-2}{2b-2}}+  |x|^{\frac{\alpha}{b-1}-1}+ |x|^{-\alpha-1} \Big]\leq  C_{\alpha,b}\Big[ |x|^{\frac{\alpha}{b-1}-\frac{b-2}{2b-2}}+ |x|^{-\alpha-1} \Big],\quad \text{for}\,\,|x|\geq 100.
\end{equation}
As a result, combining \eqref{|x|<100} with \eqref{|x|>100} yields the claimed estimate.

\end{proof}

\subsection{The fractional operator}
With exactly the same arguments, we can obtain analogous results for the operator $J^{\alpha,b}.$ Indeed, these results are easier to obtain because of the symmetry
\begin{equation*}	
    \begin{aligned}
		\int_{E} e^{i[|\xi|^b+x\xi]}|\xi|^{\alpha}d\xi =\int_{-E} e^{i[|\xi|^b-x\xi]}|\xi|^{\alpha}d\xi,
	\end{aligned}	
\end{equation*}
where $E$ is a measurable subset of the negative numbers. We have
\begin{proposition}\label{opJ}
	Let $\alpha\in  \R, b\in  \R$, and $b> 1$. Then
	\begin{equation*}
		\begin{aligned}
			|J^{b,\alpha}(x)|=\Big|\int_{\R} e^{i[|\xi|^b+x\xi]}|\xi|^{\alpha}d\xi\Big| \lesssim_{\alpha,b}\begin{cases} (1+|x|)^{-1-\alpha}  & \quad  -1<\alpha\leq  -1/2,  \\  (1+|x|)^{\frac{1}{b-1}(\alpha-\frac{b}{2}+1)}  & \quad  -1/2<\alpha<  b-1. \end{cases}
		\end{aligned}	
	\end{equation*}
\end{proposition}
The range of $\alpha$ is sharp here as well, by the examples given in the previous section. 
We scale this estimate to obtain the following corollary: 
\begin{Corollary}
	With the same $\alpha$ and $b$ as in Proposition \ref{opJ} , we have
	\begin{equation*}
		\begin{aligned}
			\Big|\int_{\R} e^{i[t|\xi|^b+x\xi]}|\xi|^{\alpha}d\xi\Big| &\lesssim_{\alpha} \begin{cases}
				(|t|^{1/b}+|x|)^{-\alpha-1} \quad  &   \quad  -1<\alpha\leq -1/2,  \\ 	(|t|^{1/b}+|x|)^{\frac{1}{b-1}(\alpha-\frac{b}{2}+1   )}|t|^{\frac{1}{b-1}(-\alpha-\frac{1}{2})} \quad  &   \quad  -1/2<\alpha<b-1.
			\end{cases}
		\end{aligned}
	\end{equation*}
\end{Corollary}
\vspace{0.2cm}
\noindent
\subsection{Proof of Proposition \ref{mainprop}} We start by considering the following case.
 \vspace{0.2cm}

\noindent
\textbf{Case 1: $p\geq q\geq 2$.} We start by applying the law of the iterated logarithm (see \cite[Theorem 8.8.3]{Durrett_Prob_Book}) to $z_{*}(t)$:
\begin{equation}\label{LIL}
	|D^{\alpha}_xz_*(t,x)|=	\Big|	\int_{t}^{\infty}V(t-r)D^{\alpha}_x\phi(x)g(r)dB(r)  \Big|  \leq  C_{\epsilon}\Big(\int_{t}^{\infty}|V(t-r)D^{\alpha}_x\phi(x)|^2g^2(r)dr  \Big)^{1/2-\epsilon},
\end{equation}
 for small $\epsilon >0$.  Thus, $2(\frac{1}{2}-\epsilon)< 1$, and  this  preclude the use of interpolation. Therefore, by the Minkowski integral inequality (as $p\geq q$) and \eqref{LIL}, we have
 \begin{equation*}
	\|D^{\alpha}z_*\|_{L^p_xL^q_t ([T,\infty))}\leq 	\|D^{\alpha}z_*\|_{L^q_tL^p_x ([T,\infty))}\leq  C_{\epsilon}\Big\|\Big[\int_{\R}\Big(\int_{t}^{\infty}|V(t-r)D^{\alpha}_x\phi(x)|^2g^2(r)dr   \Big)^{\frac{p}{2}-p\epsilon}dx\Big]^{\frac{1}{p}}\Big\|_{L^q_t ([T,\infty))} .
\end{equation*}
As $p\geq 4$ in   Kato triples, we can apply the Minkowski integral inequality again to get 
\begin{equation}\label{lqnorm}
    \begin{aligned}
    \|D^{\alpha}z_*\|_{L^p_xL^q_t ([T,\infty))}&\leq  C_{\epsilon}\Big\|\Big[\int_{t}^{\infty}\|V(t-r)D^{\alpha}_x\phi\|_{L^{p-}_x}^{2}|g(r)|^{2}dr\Big]^{\frac{1}{2}-}\Big\|_{L^q_t ([T,\infty))}.
    \end{aligned}
\end{equation}
To proceed further, we shall analyze the integral inside the $L^q_t$-norm on the right side of \eqref{lqnorm}. For this purpose, we note the dispersive inequality
\begin{equation}\label{dispereq}
	\|V(t-r)D^{\alpha}_x\phi\|_{L^{p-}_x} \leq C(t-r)^{-\frac{1}{3}(1-\frac{2}{p-})}\|D^{\alpha}_x\phi \|_{L^{(p-)'}_x}
\end{equation}
and the inequality (Hausdorff-Young inequality)
\begin{equation}\label{HY}
	\|V(t-r)D^{\alpha}_x\phi\|_{L^{p-}_x} \leq \|\widehat{D^{\alpha}_x\phi} \|_{L^{(p-)'}_x}.
\end{equation}
As $t\geq T\gg 1$, we write
\begin{equation*}
	\begin{aligned}
		\int_{t}^{\infty}\|V(t-r)D^{\alpha}_x\phi\|_{L^{p-}_x}^{2}|g(r)|^{2}dr&=\int_{t}^{t+1}+\int_{t+1}^{2t}+\int_{2t}^{\infty}\|V(t-r)D^{\alpha}_x\phi\|_{L^{p-}_x}^{2}|g(r)|^{2}dr\\ &=I_{1}+I_{2}+I_{3}.
	\end{aligned}
\end{equation*}	
Applying \eqref{HY} to the first integral gives
\begin{equation*}
	\begin{aligned}
		I_{1}\lesssim \int_{t}^{t+1} \|\widehat{D^{\alpha}_x\phi} \|_{L^{(p-)'}_x}^2t^{-2\gamma}dr \leq   t^{-2\gamma}\|\widehat{D^{\alpha}_x\phi} \|_{L^{(p-)'}_x}^2.
	\end{aligned}
\end{equation*}	
For the second, the inequality in \eqref{dispereq} implies
\begin{equation*}
	\begin{aligned}
		I_{2}\lesssim \int_{t+1}^{2t} |r-t|^{-\frac{2}{3}(1-\frac{2}{p-})}\|{D^{\alpha}_x\phi} \|_{L^{(p-)'}_x}^2t^{-2\gamma}dr &\leq   t^{-2\gamma}\|{D^{\alpha}_x\phi} \|_{L^{(p-)'}_x}^2\int_{t+1}^{2t} |r-t|^{-\frac{2}{3}(1-\frac{2}{p-})}dr  \\  &\leq   t^{-2\gamma}\|{D^{\alpha}_x\phi} \|_{L^{(p-)'}_x}^2\int_{1}^{t} r^{-\frac{2}{3}(1-\frac{2}{p-})}dr\\  &\lesssim t^{-2\gamma-\frac{2}{3}(1-\frac{2}{p-})+1}\|{D^{\alpha}_x\phi} \|_{L^{(p-)'}_x}^2  .
	\end{aligned}
\end{equation*}	
Lastly, by \eqref{dispereq}, 
\begin{equation}\label{newref1}
	\begin{aligned}
		I_{3}\lesssim \int_{2t}^{\infty} r^{-\frac{2}{3}(1-\frac{2}{p-})}\|{D^{\alpha}_x\phi} \|_{L^{(p-)'}_x}^2r^{-2\gamma}dr &\leq   \|{D^{\alpha}_x\phi} \|_{L^{(p-)'}_x}^2\int_{2t}^{\infty} r^{-\frac{2}{3}(1-\frac{2}{p-})-2\gamma}dr \\ &\lesssim t^{-\frac{2}{3}(1-\frac{2}{p-})-2\gamma+1}\|{D^{\alpha}_x\phi} \|_{L^{(p-)'}_x}^2.
	\end{aligned}
\end{equation}	
 provided that we have
\begin{equation}\label{as1}
	\frac{2}{3}\Big(1-\frac{2}{p-}\Big)+2\gamma>1 \iff 2\gamma>\frac{4}{3p-}+\frac{1}{3}
\end{equation}
for the finiteness of the last integral in \eqref{newref1}. For  $L^q_t$ integrability in \eqref{lqnorm} we further need
\begin{equation}\label{as2}
	\frac{q}{2}\big(-\frac{2}{3}(1-\frac{2}{p-})-2\gamma+1\big)<-1\iff   \frac{2}{3}(1-\frac{2}{p-})+2\gamma-1> \frac{2}{q} \iff  2\gamma> \frac{2}{q}-\frac{4}{3p-}+\frac{1}{3}.
\end{equation}
For the first condition \eqref{as1}, given that for Kato triples $p\geq q$ implies $6\leq p\leq \infty$, we need $\gamma>\frac{5}{18}$. As for the second condition \eqref{as2}, the worst case is $(p,q)=(\infty,2)$. Even then, $\gamma>2/3$ is enough to obtain \eqref{kato_type_est_tail}.

\vspace{0.2cm}
\noindent
\textbf{Case 2: $2\leq p<q$.} 
By the law of the iterated logarithm, we have
\begin{equation}\label{estV}
	\begin{aligned}
		\Big|\int_{t}^{\infty}V(t-r)D^{\alpha}_x\phi(x)g(r)dB(r)  \Big|&=\Big |\int_{t}^{\infty}\Big[\int_{\R}e^{i[(t-r)\xi^3+x\xi]}|\xi|^{\alpha}\widehat \phi(\xi)g(r)d\xi\Big]dB(r)  \Big| \\& \leq C_{\epsilon}\Big (\int_{t}^{\infty}\Big|\int_{\R}e^{i[(t-r)\xi^3+x\xi]}|\xi|^{\alpha}\widehat \phi(\xi)g(r)d\xi\Big|^2dr  \Big)^{1/2-\epsilon} \\&=C_{\epsilon}\Big (\int_{t}^{\infty}\Big(\Big|\int_{\R} e^{i[(t-r)\xi^3+x\xi]}|\xi|^{\alpha}d\xi\Big| \ast \Big| \phi(x)g(r)\Big|\Big)^2dr  \Big)^{1/2-\epsilon},
	\end{aligned}	
\end{equation}
where in the last line above, for fixed $r,t$, we use the identity
\begin{equation*}
	\int_{\R}\widehat{u}(\xi)\widehat{v}(\xi)e^{ix\xi}d\xi=u\ast v (x)
\end{equation*}
with $\widehat{u}=e^{i[(t-r)\xi^3]}|\xi|^{\alpha}$ and $\widehat{v}(\xi)=\widehat{\phi}(\xi)g(r)$. Applying Corollary \ref{remark:osc_int_time_decay} with $b=3$,
\begin{equation*}
	\Big|\int_{\R} e^{i[(t-r)\xi^3+x\xi]}|\xi|^{\alpha}d\xi\Big|\lesssim_{\alpha} |t-r|^{-\frac{\alpha}{2}-\frac{1}{4}}|x|^{\frac{2\alpha-1}{4}},
\end{equation*}
and applying this together with \eqref{estV} yields
\begin{equation}\label{newref3}
	\begin{aligned}
		&\|D^{\alpha}_xz_*(t,x)\|_{L^p_xL^q_t([T,\infty))}\\&\leq  C_{\epsilon}\Big(\int_{\R}\Big[ \int_{T}^{\infty}\Big (\int_{t}^{\infty}\Big(\Big|\int_{\R} e^{i[(t-r)\xi^3+x\xi]}|\xi|^{\alpha}d\xi\Big| \ast \Big| \phi(x)g(r)\Big|\Big)^2dr  \Big)^{q(1/2-\epsilon)} dt\Big]^{\frac{p}{q}}dx\Big)^{\frac{1}{p}} \\ &\leq  C_{\epsilon}\Big(\int_{\R}\Big[ \int_{T}^{\infty}\Big (\int_{t}^{\infty}\Big( \big[|x|^{\frac{2\alpha-1}{4}} \ast|\phi(x) |\big]\big[| g(r)||t-r|^{-\frac{2\alpha+1}{4}}\big]\Big)^2dr  \Big)^{q(1/2-\epsilon)} dt\Big]^{\frac{p}{q}}dx\Big)^{\frac{1}{p}} \\& =C_{\epsilon}\Big( \int_{T}^{\infty}\Big (\int_{t}^{\infty}| g(r)|^2|t-r|^{-\frac{2\alpha+1}{2}}dr  \Big)^{q(1/2-\epsilon)} dt\Big)^{\frac{1}{q}} \Big(\int_{\R} \Big[|x|^{\frac{2\alpha-1}{4}} \ast|\phi(x) |\Big]^{p(1-2\epsilon)}dx\Big)^{\frac{1}{p}}. 
	\end{aligned}	
\end{equation}
We observe that $|\cdot|^{\frac{2\alpha-1}{4}} \ast	\phi\sim (1+|\cdot|)^{\frac{2\alpha-1}{4}}$. As $p< q$, Kato triples satisfy 
\begin{equation}\label{newref2}
	\big(p,q,\alpha\big)=\Big(\frac{5}{1-\alpha},\frac{10}{4\alpha+1},\alpha\Big),\,\,\,\text{where}\,\,p\in [4,6),\,\,q\in (6,\infty],\,\,\alpha\in [-1/4,1/6).
\end{equation}
 Hence, integrability of $|\cdot|^{\frac{2\alpha-1}{4}} \ast|\phi(\cdot) |$ in $L^{p-}_{x}$ imposes the condition
\begin{equation}\label{palphacond}
	p\,\frac{2\alpha-1}{4}=	\frac{5}{1-\alpha} \frac{2\alpha-1}{4}<-1\quad \iff  \quad  \alpha<1/6.
\end{equation}
But as we remarked in \eqref{newref2}, this follows from our assumption $p<q$.

To estimate the first integral in the last line of \eqref{newref3}, we write
\begin{equation*}
	\begin{aligned}
		\int_{t}^{\infty}|g(r)|^2|t-r|^{-\frac{2\alpha+1}{2}}dr\sim 	\int_{t}^{2t}+\int_{2t}^{\infty}r^{-2\gamma}|t-r|^{-\frac{2\alpha+1}{2}}dr= I+II.
	\end{aligned}	
\end{equation*}
Then 
\begin{equation*}
	\begin{aligned}
	I\lesssim	t^{-2\gamma}	\int_{t}^{2t}|t-r|^{-\frac{2\alpha+1}{2}}dr=t^{-2\gamma}	\int_{0}^{t}r^{-\frac{2\alpha+1}{2}}dr\sim t^{-2\gamma-\alpha+\frac{1}{2}}.
	\end{aligned}	
\end{equation*}
The last integral in this equation imposes the condition $\alpha<1/2$.  As for $II$
\begin{equation*}
	\begin{aligned}
	II=\int_{2t}^{\infty}r^{-2\gamma}|t-r|^{-\frac{2\alpha+1}{2}}dr\sim \int_{2t}^{\infty}r^{-2\gamma-\frac{2\alpha+1}{2}}dr\sim t^{-2\gamma-\alpha+\frac{1}{2}},
	\end{aligned}	
\end{equation*}
where the last integral demands the condition $4\gamma+2\alpha >1$. Finally, we need integrability of $t^{-2\gamma-\alpha+\frac{1}{2}}$ raised to power $\frac{q}{2}-$ on $[T,\infty)$. This leads to the condition
\begin{equation*}
\Big(-2\gamma-\alpha+\frac{1}{2}\Big)\frac{q}{2}<-1 \implies 4\gamma+2\alpha >1+\frac{4}{q}
\end{equation*}
which supersedes the above condition. For Kato triples $(p,q,\alpha)$ with $p\leq q$, these clearly hold. For Kato triples 
 \begin{equation*}
	\alpha=\frac{5}{2q}-\frac{1}{4} \implies 4\gamma >\frac{3}{2}-\frac{1}{q}
\end{equation*}
meaning even $\gamma>3/8$ is sufficient.

To sum up, for all Kato triples, our assumption $\gamma>2/3$ is sufficient to obtain the estimates claimed in the Proposition \ref{mainprop}.

\section{Scattering in $L^{2}_{x}(\mathbb{R})$ and $H_x^1(\mathbb{R})$}\label{section:scat_in_L2}
\subsection{Proof of Theorem \ref{thm:scat_in_L2}}\label{pfthm1.3}
Consider the random IVP \eqref{eq:differ_rand_and_det_soln_IVP} with $k=4$ and the Duhamel representation of solutions
\begin{equation}\label{duhamel_of_v_mass_crit}
    v(t)=\int_{T}^{t}V(t-t')\left[(( v + y + z_{*})^{5})_{x} - (y^{5})_{x}\right](t')dt'\quad t\in[T,\infty),
\end{equation}
where $T\gg 1$ is a large time parameter. Using \eqref{duhamel_of_v_mass_crit}, Proposition \ref{Kato_est_most_general} and Hölder inequality, we have
\begin{align*}
    \norm{v}_{L^{5}_{x}L^{10}_{t}([T,\infty))}\lesssim&\norm{\partial_{x}^{-1}[\partial_{x}((v+y+z_{*})^{5}-y^{5})]}_{L^{1}_{x}L^{2}_{t}([T,\infty))}\\
    \lesssim&C(T)(1+\norm{v}_{L^{5}_{x}L^{10}_{t}([T,\infty))}^{4})+\norm{v}_{L^{5}_{x}L^{10}_{t}([T,\infty))}^{5}\\
    \lesssim& C(T)+(1+C(T))\norm{v}_{L^{5}_{x}L^{10}_{t}([T,\infty))}^{5},
\end{align*}
where we set $C(T):=C(\norm{y}_{L^{5}_{x}L^{10}_{t}([T,\infty))},\norm{z_{*}}_{L^{5}_{x}L^{10}_{t}([T,\infty))}))$. Therefore, applying Lemma \ref{lemma:uniform_bounded_by_its_higher_powers} yields
\begin{equation}\label{decay_of_scattering_size_of_v}
    \norm{v}_{L^{5}_{x}L^{10}_{t}([T,\infty))}\lesssim C(T)\to 0\text{ as }T\to\infty,
\end{equation}
where the above limit holds $\omega$-a.s. thanks to the global-space-time bounds for $y(t)$ and $z_{*}(t)$ (in view of the results (\cite[Theorem 1.2]{Dodson_gKdV} for the deterministic solution $y(t)$ and Proposition \ref{mainprop} with the triple $(\alpha,p,q)=(0,5,10)$ for $z_{*}(t)$. In addition, we similarly have
\begin{equation}
\begin{aligned}\label{frstnormv}
    \sup\limits_{t\in[T,\infty)}\norm{v(t)}_{L^{2}_{x}}\lesssim&\norm{(v+y+z_{*})^{5}-y^{5}}_{L^{1}_{x}L^{2}_{t}([T,\infty))}\\
    \lesssim&\norm{\vert v+z_{*}\vert(\vert v+y+z_{*}\vert^{4}+\vert y\vert^{4})}_{L^{1}_{x}L^{2}_{t}([T,\infty))}\\
    \lesssim& C(T)(1+\norm{v}_{L^{5}_{x}L^{10}_{t}([T,\infty))}^{4})+\norm{v}_{L^{5}_{x}L^{10}_{t}([T,\infty))}^{5},
\end{aligned}
\end{equation}
and
\begin{equation}
\begin{aligned}\label{secondnormv}
\sup\limits_{x\in\mathbb{R}}\norm{\partial_{x}v(x)}_{L^{2}_{t}([T,\infty))}
    \lesssim C(T)(1+\norm{v}_{L^{5}_{x}L^{10}_{t}([T,\infty))}^{4})+\norm{v}_{L^{5}_{x}L^{10}_{t}([T,\infty))}^{5}.
\end{aligned}
\end{equation}
Thus, using the global boundedness of $L^{5}_{x}L^{10}_{t}$ norm of $v$ due to \eqref{decay_of_scattering_size_of_v}, the right sides of \eqref{frstnormv} and \eqref{secondnormv} tend to $0$ as $T\to\infty$.

It remains to show that $\{V(-t)u_{*}(t):\, t\in[T,\infty)\}$ is Cauchy in $L^{2}_{x}(\mathbb{R})$. Let $T<t_{1}<t_{2}<\infty$. Thus,
\begin{align*}
    \Vert V(-t_{2})u_{*}(t_{2})-V(-t_{1})u_{*}(t_{1})\Vert_{L^{2}_{x}}&\leq\Vert V(-t_{2})(u_{*}(t_{2})-y(t_{2}))\Vert_{L^{2}_{x}}+\Vert V(-t_{2})y(t_{2})-V(-t_{1})y(t_{1})\Vert_{L^{2}_{x}}\\
    &\quad+\Vert V(-t_{1})(u_{*}(t_{1})-y(t_{1}))\Vert_{L^{2}_{x}}\\
    &\lesssim\sup\limits_{t\in[T,\infty)}\Vert v(t)\Vert_{L^{2}_{x}}+ \Vert V(-t_{2})y(t_{2})-V(-t_{1})y(t_{1})\Vert_{L^{2}_{x}}\\
    &\lesssim C(T)+\Vert V(-t_{2})y(t_{2})-V(-t_{1})y(t_{1})\Vert_{L^{2}_{x}}\to 0\quad \text{as $T\to\infty$,}
\end{align*}
since $y$ scatters in $L^{2}_{x}(\mathbb{R})$ (which is known by \cite[Theorems 1.1 and 1.2]{Dodson_gKdV}) and $v(t)\to 0$ in $L^{2}_{x}(\mathbb{R})$ as $t\to\infty$ due to \eqref{frstnormv}. This shows that $u_{*}(t)$, $\omega$ almost surely, converges to the same strong limit of $y(t)$ in $L^{2}_{x}(\mathbb{R})$ as $t\to\infty$, i.e.,
\begin{equation*}
    \lim_{t\to\infty}\Vert V(-t)u_{*}(t)-y_{+}\Vert_{L^{2}_{x}}=0,\quad \omega\text{-a.s.},
\end{equation*}
where $y_{+}$ is the strong limit of $V(-t)y(t)$ in $L^{2}_{x}(\mathbb{R})$ as $t\to\infty$.

\subsection{Proof of Theorem \ref{thm:scat_in_H1}}\label{section:scat_in_H1}
In this subsection, we always assume $k>4$, and will follow the notations and notion of decomposition of solutions \eqref{Duhamel_of_stochastic_solution} introduced in Section \ref{tss} in this range. We initially review the related result in \cite{Farah_H1_scattering} ($H^{1}_{x}$ scattering for the solution $y(t)$ of \eqref{eq:deterministic_gKdV_article}):
\begin{proposition}[Proposition 3.3 in \cite{Farah_H1_scattering}]\label{propfarah}
    Let $k>4$ be an even integer. If $y(T)\in H^{1}_{x}(\mathbb{R})$ and
    \begin{equation*}
        y(t)=V(t-T)y(T)+\int_{T}^{t}V(t-t')(y^{k+1})_{x}(t')dt'
    \end{equation*}
    is a global solution satisfying $\sup\limits_{t\in[T,\infty)}\Vert y(t)\Vert_{H^{1}_{x}}<\infty$ and $\Vert y\Vert_{L^{5k/4}_{x}L^{5k/2}_{t}([T,\infty))}<\infty$, then there exists $\phi^{+}\in H^{1}_{x}$ such that
    \begin{equation*}
        \lim_{t\to\infty}\Vert V(-t)y(t)-\phi^{+}\Vert_{H^{1}_{x}}=0.
    \end{equation*}
\end{proposition}
The above proposition illustrates that the $L^{\frac{5k}{4}}_{x}L^{\frac{5k}{2}}_{t}$ norm measures the scattering size of solutions to the mass supercritical equations \eqref{eq:stochastic_gKdV}. Indeed, proving the global in time boundedness of this norm is essential for energy scattering. In what follows, utilizing the results of Propositions \ref{mainprop} and \ref{propfarah}, we will demonstrate this boundedness for $v$ in $\omega$ almost sure sense.
For this purpose, we consider the Duhamel representation
\begin{equation}\label{duhamel_of_v_mass_supercrit}
    v(t)=\int_{T}^{t}V(t-t')\left[(( v + y + z_{*})^{k+1})_{x} - (y^{k+1})_{x}\right](t')dt'\quad t\in[T,\infty)
\end{equation}
and define
\begin{equation*}
    \begin{aligned}
        \lambda_{1}(v)=&\norm{\langle\partial_{x}\rangle v}_{L^{5}_{x}L^{10}_{t}([T,\infty))},\\
        \lambda_{2}(v)=&\norm{v}_{L^{\frac{5k}{4}}_{x}L^{\frac{5k}{2}}_{t}([T,\infty))}.
    \end{aligned}
\end{equation*}
Set $\Lambda(v)=\max\{\lambda_{1}(v),\lambda_{2}(v)\}$. 
Then, using \eqref{duhamel_of_v_mass_supercrit}, Proposition \ref{Kato_est_most_general} and Hölder's inequality, we get
\begin{equation}\label{lambda1}
\begin{aligned}
    \lambda_{1}(v)=&\norm{\langle\partial_{x}\rangle\int_{T}^{t}V(t-t')\partial_{x}\left((v+y+z_{*})^{k+1}-y^{k+1}\right)(t')dt'}_{L^{5}_{x}L^{10}_{t}([T,\infty))}\\
    \lesssim&\norm{\langle\partial_{x}\rangle\left((v+y+z_{*})^{k+1}-y^{k+1}\right)}_{L^{1}_{x}L^{2}_{t}([T,\infty))}\\
    \lesssim&\norm{\langle\partial_{x}\rangle(v+z_{*})}_{L^{5}_{x}L^{10}_{t}([T,\infty))}\norm{v+y+z_{*}}_{L^{\frac{5k}{4}}_{x}L^{\frac{5k}{2}}_{t}([T,\infty))}^{k}\\
    &+\norm{\langle\partial_{x}\rangle y}_{L^{5}_{x}L^{10}_{t}([T,\infty))}\left(\norm{v+y+z_{*}}_{L^{\frac{5k}{4}}_{x}L^{\frac{5k}{2}}_{t}([T,\infty))}^{k}+\norm{y}_{L^{\frac{5k}{4}}_{x}L^{\frac{5k}{2}}_{t}([T,\infty))}^{k}\right)\\
    \lesssim&(\lambda_{1}(v)+\lambda_{1}(z_{*}))\left(\lambda_{2}^{k}(v)+\lambda_{2}^{k}(y)+\lambda_{2}^{k}(z_{*})\right)+\lambda_{1}(y)\left(\lambda_{2}^{k}(v)+\lambda_{2}^{k}(y)+\lambda_{2}^{k}(z_{*})\right)\\
    \lesssim& C(T,y,z_{*})\left(1+\Lambda(v)+\Lambda^{k}(v)+\Lambda^{k+1}(v)\right)\\
    \lesssim& C(T,y,z_{*})(1+\Lambda^{k+1}(v))
\end{aligned}
\end{equation}
where $C(T,y,z_{*})>0$ is given by $C(\Lambda(y)+\Lambda(z_{*}))^{k+1}$ for some absolute constant $C>0$,
which tends to $0$ as $T\to\infty$ in virtue of Propositions \ref{mainprop} and \ref{propfarah}.
As for $\lambda_{2}(v)$, applying Sobolev embedding in $x$ and $t$, Proposition \ref{Kato_est_most_general} and the Hölder's inequality gives
\begin{equation}\label{lambda2}
\begin{aligned}
    \lambda_{2}(v)=&\norm{\partial_{x}\int_{T}^{t}V(t-t')\left((v+y+z_{*})^{k+1}-y^{k+1}\right)(t')dt'}_{L^{\frac{5k}{4}}_{x}L^{\frac{5k}{2}}_{t}([T,\infty))}\\
    \lesssim&\norm{D_{x}^{\frac{1}{2}-\frac{2}{k}}\left((v+y+z_{*})^{k+1}-y^{k+1}\right)}_{L^{1}_{x}L^{2}_{t}([T,\infty))}\\
    \lesssim&\norm{\langle\partial_{x}\rangle\left((v+y+z_{*})^{k+1}-y^{k+1}\right)}_{L^{1}_{x}L^{2}_{t}([T,\infty))}\\
    \leq& C(T,y,z_{*})(1+\Lambda^{k+1}(v)).
\end{aligned}
\end{equation}
We now want to estimate $\norm{\partial_{x}^{2}v}_{C^{0}_{x}L^{2}_{t}([T,\infty))}$ and $\norm{v}_{C^{0}_{t}H^{1}_{x}([T,\infty))}$. For the former, Kato estimate \eqref{est:Kato_3} leads to
\begin{equation}\label{der2v}
\begin{aligned}
   \norm{\partial_{x}^{2}v}_{C^{0}_{x}L^{2}_{t}([T,\infty))}=&\norm{\partial_{x}^{2}\int_{T}^{t}V(t-t')\partial_{x}\left((v+y+z_{*})^{k+1}-y^{k+1}\right)(t')dt'}_{C^{0}_{x}L^{2}_{t}([T,\infty))}\\
   \lesssim&\norm{\partial_{x}\left((v+y+z_{*})^{k+1}-y^{k+1}\right)}_{L^{1}_{x}L^{2}_{t}([T,\infty))}\\
   \leq&C(T,y,z_{*})(1+\Lambda^{k+1}(v)).
\end{aligned}
\end{equation}
For the latter, Kato estimate \eqref{est:Kato_2} gives
\begin{equation}\label{v}
\begin{aligned}
    \norm{v}_{C^{0}_{t}H^{1}_{x}([T,\infty))}=&\sup\limits_{T\leq t<\infty}\norm{\langle\partial_{x}\rangle\int_{T}^{t}V(t-t')\partial_{x}\left((v+y+z_{*})^{k+1}-y^{k+1}\right)(t')dt'}_{L^{2}_{x}}\\
    =&\sup\limits_{T\leq t<\infty}\norm{\langle\partial_{x}\rangle\int_{T}^{t}V(-t')\partial_{x}\left((v+y+z_{*})^{k+1}-y^{k+1}\right)(t')dt'}_{L^{2}_{x}}\\
    \lesssim&\norm{\langle\partial_{x}\rangle\left((v+y+z_{*})^{k+1}-y^{k+1}\right)}_{L^{1}_{x}L^{2}_{t}([T,\infty))}\\
    \leq& C(T,y,z_{*})(1+\Lambda^{k+1}(v)).
\end{aligned}
\end{equation}
Therefore, from \eqref{lambda1}--\eqref{v}, we have
\begin{equation}\label{bounds_for_crucial_norms_H_1}
    \Lambda(v)+\norm{\partial_{x}^{2}v}_{C^{0}_{x}L^{2}_{t}([T,\infty))}+\norm{v}_{C^{0}_{t}H^{1}_{x}([T,\infty))}\leq C(T,y,z_{*})(1+\Lambda^{k+1}(v)).
\end{equation}
Now applying Lemma \ref{lemma:uniform_bounded_by_its_higher_powers} to \eqref{bounds_for_crucial_norms_H_1}, we arrive at
\begin{equation}\label{decay_of_norms_of_v}
    \Lambda(v)+\norm{\partial_{x}^{2}v}_{C^{0}_{x}L^{2}_{t}([T,\infty))}+\norm{v}_{C^{0}_{t}H^{1}_{x}([T,\infty))}\lesssim C(T,y,z_{*})\to 0\text{ as }T\to\infty.
\end{equation}
Finally, it is left to prove that $\{V(-t)u_{*}(t):\,t\in[T,\infty)\}$ is Cauchy in $H^{1}_{x}(\mathbb{R})$. Take $T<t_{1}<t_{2}<\infty$. Then, $\omega$-a.s., we have
\begin{align*}
    \norm{V(-t_{2})u_{*}(t_{2})-V(-t_{1})u_{*}(t_{1})}_{H^{1}_{x}}\leq&\norm{V(-t_{2})(u_{*}(t_{2})-y(t_{2}))}_{H^{1}_{x}}+\norm{V(-t_{2})y(t_{2})-V(-t_{1})y(t_{1})}_{H^{1}_{x}}\\
    &+\norm{V(-t_{1})(u_{*}(t_{1})-y(t_{1}))}_{H^{1}_{x}}\\
    \lesssim&\sup\limits_{T\leq t<\infty}\norm{v(t)}_{H^{1}_{x}}+\norm{V(-t_{2})y(t_{2})-V(-t_{1})y(t_{1})}_{H^{1}_{x}}\to 0\text{ as }T\to \infty
\end{align*}
since $y$ scatters in $H^{1}_{x}(\mathbb{R})$ by Proposition \ref{propfarah} and $v(t)\to 0$ in $H^{1}_{x}(\mathbb{R})$ as $t\to\infty$ due to \eqref{decay_of_norms_of_v}. Therefore, $u_{*}(t)$ converges to the same strong limit of $y(t)$ in $H^{1}_{x}(\mathbb{R})$ as $t\to\infty$, i.e.,
\begin{equation*}
    \lim_{t\to\infty}\norm{V(-t)u_{*}(t)-y_{+}}_{H^{1}_{x}}=0,\quad \omega\text{-a.s.},
\end{equation*}
where $y_{+}$ is the strong limit of $V(-t)y(t)$ in $H^{1}_{x}(\mathbb{R})$ as $t\to\infty$.

\nocite{*}
\bibliographystyle{abbrv}
\bibliography{Reference.bib}

\begin{thebibliography}{10}

\bibitem{Basakoglu_Scattering_NLS}
E.~Başakoğlu, F.~Temur, B.~Yeşiloğlu, and O.~Yılmaz.
\newblock Scattering for stochastic nonlinear {S}chr\"odinger equations with
  additive noise, 2024.
\newblock https://arxiv.org/abs/2412.03469.

\bibitem{cheung2020conservation}
K.~Cheung, G.~Li, and T.~Oh.
\newblock Almost conservation laws for stochastic nonlinear {S}chr\"{o}dinger
  equations.
\newblock {\em Journal of Evolution Equations}, 21(2):1865–1894, Jan. 2021.

\bibitem{colliander2002almost}
J.~Colliander, M.~Keel, G.~Staffilani, H.~Takaoka, and T.~Tao.
\newblock Almost conservation laws and global rough solutions to a nonlinear
  {S}chr\"{o}dinger equation.
\newblock {\em Mathematical Research Letters}, 9(5):659–682, 2002.

\bibitem{Da_Prato_book}
G.~Da~Prato and J.~Zabczyk.
\newblock {\em Stochastic equations in infinite dimensions}, volume 152 of {\em
  Encyclopedia of Mathematics and its Applications}.
\newblock Cambridge University Press, Cambridge, second edition, 2014.

\bibitem{deBouard_SKdV_1}
A.~de~Bouard and A.~Debussche.
\newblock On the stochastic {K}orteweg-de {V}ries equation.
\newblock {\em J. Funct. Anal.}, 154(1):215--251, 1998.

\bibitem{deBouard_H1}
A.~de~Bouard and A.~Debussche.
\newblock The stochastic nonlinear {S}chrödinger equation in ${H}^{1}$.
\newblock {\em Stochastic Analysis and Applications}, 21(1):97--126, 2003.

\bibitem{deBouard_SKdV_3}
A.~de~Bouard and A.~Debussche.
\newblock The {K}orteweg-de {V}ries equation with multiplicative homogeneous
  noise.
\newblock In {\em Stochastic differential equations: theory and applications},
  volume~2 of {\em Interdiscip. Math. Sci.}, pages 113--133. World Sci. Publ.,
  Hackensack, NJ, 2007.

\bibitem{deBouard_SKdV_2}
A.~de~Bouard, A.~Debussche, and Y.~Tsutsumi.
\newblock White noise driven {K}orteweg-de {V}ries equation.
\newblock {\em J. Funct. Anal.}, 169(2):532--558, 1999.

\bibitem{Dodson_gKdV}
B.~Dodson.
\newblock Global well-posedness and scattering for the defocusing,
  mass-critical generalized {K}d{V} equation.
\newblock {\em Ann. PDE}, 3(1):Paper No. 5, 35, 2017.

\bibitem{Durrett_Prob_Book}
R.~Durrett.
\newblock {\em Probability: theory and examples}, volume~31 of {\em Cambridge
  Series in Statistical and Probabilistic Mathematics}.
\newblock Cambridge University Press, Cambridge, fourth edition, 2010.

\bibitem{ERDOGAN_PDE_Book}
M.~Erdoğan and N.~Tzirakis.
\newblock {\em Dispersive partial differential equations}, volume~86 of {\em
  London Mathematical Society Student Texts}.
\newblock Cambridge University Press, Cambridge, 2016.
\newblock Wellposedness and applications.

\bibitem{Fan_Zhao}
C.~Fan and Z.~Zhao.
\newblock On long time behavior for stochastic nonlinear {S}chr\"odinger
  equations with a multiplicative noise.
\newblock {\em Int. Math. Res. Not. IMRN}, (10):8882--8904, 2024.

\bibitem{Farah_critical_gKdV}
L.~Farah.
\newblock Global rough solutions to the critical generalized {K}d{V} equation.
\newblock {\em J. Differential Equations}, 249(8):1968--1985, 2010.

\bibitem{Farah_supercrit_gKdV}
L.~Farah, F.~Linares, and A.~Pastor.
\newblock The supercritical generalized {K}d{V} equation: global well-posedness
  in the energy space and below.
\newblock {\em Math. Res. Lett.}, 18(2):357--377, 2011.

\bibitem{Farah_H1_scattering}
L.~Farah, F.~Linares, A.~Pastor, and N.~Visciglia.
\newblock Large data scattering for the defocusing supercritical generalized
  {K}d{V} equation.
\newblock {\em Communications in Partial Differential Equations},
  43(1):118--157, 2018.

\bibitem{Grafakos_book}
L.~Grafakos.
\newblock {\em Classical {F}ourier analysis}, volume 249 of {\em Graduate Texts
  in Mathematics}.
\newblock Springer, New York, third edition, 2014.

\bibitem{Herr_2019}
S.~Herr, M.~Röckner, and D.~Zhang.
\newblock Scattering for stochastic nonlinear {S}chrödinger equations.
\newblock {\em Communications in Mathematical Physics}, 368(2):843–884, Apr.
  2019.

\bibitem{johnson_KdV_physical}
M.~Johnson and K.~Zumbrun.
\newblock Rigorous justification of the whitham modulation equations for the
  generalized {K}orteweg-de {V}ries equation, 2009.
\newblock https://arxiv.org/abs/0910.1617.

\bibitem{Kenig_gKdV_1993}
C.~Kenig, G.~Ponce, and L.~Vega.
\newblock Well-posedness and scattering results for the generalized
  {K}orteweg-de {V}ries equation via the contraction principle.
\newblock {\em Comm. Pure Appl. Math.}, 46(4):527--620, 1993.

\bibitem{Korteweg_First_article}
D.~J. Korteweg and G.~de~Vries~and.
\newblock Xli. on the change of form of long waves advancing in a rectangular
  canal, and on a new type of long stationary waves.
\newblock {\em The London, Edinburgh, and Dublin Philosophical Magazine and
  Journal of Science}, 39(240):422--443, 1895.

\bibitem{Linares_book}
F.~Linares and G.~Ponce.
\newblock {\em Introduction to nonlinear dispersive equations}.
\newblock Universitext. Springer, New York, second edition, 2015.

\bibitem{Millet_SKdV}
A.~Millet and S.~Roudenko.
\newblock Generalized {K}d{V} equation subject to a stochastic perturbation.
\newblock {\em Discrete Contin. Dyn. Syst. Ser. B}, 23(3):1177--1198, 2018.

\bibitem{Miura_KdV_integrability}
R.~Miura, C.~Gardner, and M.~Kruskal.
\newblock Korteweg-de {V}ries equation and generalizations. {II}. {E}xistence
  of conservation laws and constants of motion.
\newblock {\em J. Mathematical Phys.}, 9:1204--1209, 1968.

\bibitem{Oh_Critical}
T.~Oh and M.~Okamoto.
\newblock On the stochastic nonlinear {S}chr{\"o}dinger equations at critical
  regularities.
\newblock {\em Stochastics and Partial Differential Equations: Analysis and
  Computations}, 8(4):869--894, Dec 2020.

\bibitem{Xicheng_Zhang_Gronwall}
M.~R\"ockner and X.~Zhang.
\newblock Well-posedness of distribution dependent {SDE}s with singular drifts.
\newblock {\em Bernoulli}, 27(2):1131--1158, 2021.

\bibitem{saanouni2015remarks}
T.~Saanouni.
\newblock Remarks on damped fractional {S}chr{\"o}dinger equation with pure
  power nonlinearity.
\newblock {\em Journal of Mathematical Physics}, 56(6), 2015.

\bibitem{Segur_Hammack_1982}
H.~Segur and J.~L. Hammack.
\newblock Soliton models of long internal waves.
\newblock {\em Journal of Fluid Mechanics}, 118:285–304, 1982.

\end{thebibliography}

\end{document}